\documentclass[12pt,a4paper,reqno]{amsart}

\usepackage{tikz}
\usetikzlibrary{matrix,arrows,calc,snakes,patterns,decorations.markings}

\usepackage[mathscr]{eucal}
\usepackage{graphics,epic}
\usepackage{amsfonts}
\usepackage{amscd}
\usepackage{latexsym}
\usepackage{amsmath,amssymb, amsthm, stmaryrd, bm, bbm}
\usepackage[all,2cell]{xy}
\usepackage{mathrsfs}

\newtheorem{theorem}{Theorem}[section]
\newtheorem*{theorem*}{Theorem}

\newtheorem{lemma}[theorem]{Lemma}
\newtheorem{proposition}[theorem]{Proposition}
\newtheorem{corollary}[theorem]{Corollary}

\newtheorem*{conjecture*}{Conjecture}
\newtheorem{question}[theorem]{Question}
\theoremstyle{remark}
\newtheorem{remark}[theorem]{Remark}

\theoremstyle{definition}

\newcommand{\ie}{{\em i.e.~}\ }

\newcommand{\opname}[1]{\operatorname{\mathsf{#1}}}

\renewcommand{\mod}{\opname{mod}\nolimits}

\newcommand{\proj}{\opname{proj}\nolimits}

\newcommand{\Mod}{\opname{Mod}\nolimits}

\newcommand{\add}{\opname{add}\nolimits}
\newcommand{\op}{^{op}}

\renewcommand{\Im}{\opname{Im}\nolimits}

\newcommand{\fpr}{\opname{pr}\nolimits}

\newcommand{\cok}{\opname{cok}\nolimits}

\newcommand{\thick}{\opname{thick}\nolimits}

\newcommand{\per}{\opname{per}\nolimits}

\newcommand{\id}{\mathrm{id}}

\newcommand{\Hom}{\opname{Hom}}
\newcommand{\End}{\opname{End}}

\newcommand{\Ext}{\opname{Ext}}

\newcommand{\ca}{{\mathcal A}}

\newcommand{\cc}{{\mathcal C}}
\newcommand{\cd}{{\mathcal D}}

\newcommand{\cf}{{\mathcal F}}

\newcommand{\ch}{{\mathcal H}}
\newcommand{\ci}{{\mathcal I}}

\newcommand{\cm}{{\mathcal M}}
\newcommand{\cn}{{\mathcal N}}

\newcommand{\ct}{{\mathcal T}}

\newcommand{\del}{\partial}

\newenvironment{psmallmatrix}
  {\left(\begin{smallmatrix}}
  {\end{smallmatrix}\right)}
\numberwithin{equation}{section}

\begin{document}

\title[From finitely presented objects to 2-term complexes]{From objects finitely presented by a rigid object in a triangulated category to 2-term complexes}

\author{Dong Yang}
\address{Dong Yang, School of Mathematics, Nanjing University, Nanjing 210093, P. R. China}
\email{yangdong@nju.edu.cn}
\thanks{The author acknowledges financial support by the National Natural Science Foundation of China No.12031007. He thanks Yann Palu for very helpful conversations, especially for pointing out the two references \cite{FGPPP23} and \cite{ChenXF23}. He thanks Bernhard Keller for pointing out the references~\cite{Krause20,HuaKeller24}. He thanks Osamu Iyama for writing the appendix.}
\dedicatory{Dedicated to the memory of Xianke Zhang}

\address{Osamu Iyama, Graduate School of Mathematical Sciences, The University of Tokyo, 3-8-1 Komaba
Meguro-ku Tokyo 153-8914, Japan}
\email{iyama@ms.u-tokyo.ac.jp}

\date{\today}

\begin{abstract} For a rigid object $M$ in an algebraic triangulated category $\ct$, a functor $\fpr(M)\to\ch^{[-1,0]}(\proj A)$ is constructed, which essentially takes an object to its `presentation', where $\fpr(M)$ is the full subcategory of $\ct$ of objects finitely presented by $M$, $A$ is the endomorphism algebra of $M$ and $\ch^{[-1,0]}(\proj A)$ is the homotopy category of complexes of finitely projective $A$-modules concentrated in degrees $-1$ and $0$. This functor is shown to be full and dense and its kernel is described. It detects isomorphisms, indecomposability and extriangles. In the Hom-finite case, it induces a bijection from the set of isomorphism classes of basic relative cluster-tilting objects of $\fpr(M)$ to that of basic silting complexs of $\ch^{[-1,0]}(\proj A)$, which commutes with mutations.  These results are applied to cluster categories of self-injective quivers with potential to recover a theorem of Mizuno on the endomorphism algebras of certain 2-term silting complexes. As an interesting consequence of the main results, if $\ct$ is a 2-Calabi--Yau triangulated category and $M$ is a cluster-tilting object such that $A$ is self-injective, then $\mathbb{P}$ is an equivalence, in particular, $\ch^{[-1,0]}(\proj A)$ admits a triangle structure. In the appendix by Iyama it is shown that for a finite-dimensional algebra $A$, if $\ch^{[-1,0]}(\proj A)$ admits a triangle structure, then $A$ is necessarily self-injective.\\
{\bf MSC 2020 classification:} 16E35, 18G80.\\
{\bf Keywords:} rigid object, relative cluster-tilting object, 2-term silting complex, mutation, self-injective quiver with potential. 
\end{abstract}

\maketitle

\tableofcontents

\section{Introduction}

Let $k$ be a commutative ring and $\ct$ be an algebraic triangulated $k$-category with shift functor $\Sigma$. Let $M$ be a rigid object of $\ct$, \ie the $k$-module $\Hom_\ct(M,\Sigma M)$ of morphisms from $M$ to $\Sigma M$ is trivial. Consider the full subcategory $\fpr(M)$ of $\ct$ consisting of objects finitely presented by $M$, that is, objects $X$ with triangles $\xymatrix@C=0.7pc{M'\ar[r] & M''\ar[r] & X\ar[r] & \Sigma M'}$, where $M'$ and $M''$ are direct summands of finite direct sums of copies of $M$. Let $A=\End_\ct(M)$ be the endomorphism algebra of $M$. In the literature great attention has been paid to the functor 
\[
\Hom_\ct(M,?)\colon \fpr(M)\longrightarrow\mod A
\]
which relates cluster-tilting theory in $\fpr(M)$ and support $\tau$-tilting theory in $\mod A$, the category of finitely presented $A$-modules, see for example \cite{BuanMarshReiten07,KellerReiten07,IyamaYoshino08,KoenigZhu08,AdachiIyamaReiten14,IyamaJorgensenYang14,YangZhu19,FuGengLiu19,ZhouZhu20}. 

In this paper we construct by hand a $k$-linear functor through which the above functor naturally factors 
\[
\mathbb{P}\colon \fpr(M)\longrightarrow\ch^{[-1,0]}(\proj A)
\]
which essentially takes an object of $\fpr(M)$ to its `presentation', where $\ch^{[-1,0]}(\proj A)$ is the full subcategory of the bounded homotopy category $\ch^b(\proj A)$ of finitely generated projective $A$-modules consisting of complexes which are homotopy equivalent to complexes concentrated in degrees $-1$ and $0$. The precise construction is given in Section~\ref{s:from-finitely-presented-objects-to-two-term-complexes}. Both $\fpr(M)$ and $\ch^{[-1,0]}(\proj A)$ carry extriangle structures in the sense of \cite{NakaokaPalu19}:
\begin{itemize}
\item[-] a triangle $\xymatrix@C=0.7pc{X\ar[r] & Y\ar[r] & Z\ar[r]^(0.4)w & \Sigma X}$ with $X,Y,Z\in\fpr(M)$ is an extriangle of $\fpr(M)$ if $w$ factors through  a finite direct sum of copies of $\Sigma M$;
\item[-] a triangle $\xymatrix@C=0.7pc{X\ar[r] & Y\ar[r] & Z\ar[r] & \Sigma X}$ with $X,Y,Z\in\ch^{[-1,0]}(\proj A)$ is an extriangle of $\ch^{[-1,0]}(\proj A)$.
\end{itemize}
Our main result is

\begin{theorem}[{Theorem~\ref{thm:the-one-object-case} and Proposition~\ref{prop:P-detects-extriangles}}]
\label{thm:main-1}
\mbox{}

\begin{itemize}
\item[(a)] The functor $\mathbb{P}$ is full, dense and its morphism kernel $\ci$ consists of morphisms which factors through a morphism $\Sigma M_1\to M_2$ with $M_1,M_2\in\add(M)$. In particular, it induces an equivalence $\frac{\fpr(M)}{\ci}\to \ch^{[-1,0]}(\proj A)$, and $\mathbb{P}$ is an equivalence if and only if $\Hom_\ct(M,\Sigma^{-1}M)=0$.
\item[(b)] $\mathbb{P}$ detects isomorphisms, detects indecomposability and induces a bijection between the isomorphism classes of objects of $\fpr(M)$ and those of $\ch^{[-1,0]}(\proj A)$.
\item[(c)] $\mathbb{P}$ detects extriangles. Precisely, if $\xymatrix@=0.7pc{Y\ar[r]^u & Z\ar[r]^v & X\ar[r]^(0.4)w &\Sigma Y}$ is an extriangle in $\fpr(M)$, then its image $\xymatrix@=1.2pc{\mathbb{P}(Y)\ar[r]^{\mathbb{P}(u)} & \mathbb{P}(Z)\ar[r]^{\mathbb{P}(v)} & \mathbb{P}(X)\ar[r]^{\mathbb{P}(w)} &\Sigma \mathbb{P}(Y)}$ is an extriangle in $\ch^{[-1,0]}(\proj A)$; up to isomorphism every extriangle in $\ch^{[-1,0]}(\proj A)$ can be uniquely obtained in this way. 
\item[(d)] $\mathbb{P}$ induces a bijection between the set of isomorphism classes of  relative cluster-tilting objects of $\fpr(M)$ and the set of isomorphism classes of 2-term silting objects of $\ch^b(\proj A)$. 
\item[(e)] Assume further that $k$ is a field and $\fpr(M)$ is Hom-finite over $k$. Then the bijection in (c) restricts to a bijection between the set of isomorphism classes of basic relative cluster-tilting objects of $\fpr(M)$ and the set of isomorphism classes of basic 2-term silting objects of $\ch^b(\proj A)$, which commutes with mutations.
\end{itemize}
\end{theorem}

Theorem~\ref{thm:main-1}(a)(b) generalises \cite[Proposition A.5]{BruestleYang13}, (d) generalises \cite[Proposition A.3]{BruestleYang13} and (e) generalises \cite[Proposition A.6]{BruestleYang13}. The statements (a)(b)(c)(d) are stated more generally for rigid subcategories, see Proposition~\ref{prop:from-fpr-to-2term}, Proposition~\ref{prop:P-detects-extriangles}, Theorem~\ref{thm:the-kernel}, Corollary~\ref{cor:when-P-is-an-equivalence} and Theorem~\ref{thm:P-induces-bijection-between-cto}. We remark that (d) is compatible with the bijection established in \cite{AdachiIyamaReiten14,IyamaJorgensenYang14,YangZhu19,FuGengLiu19,ZhouZhu20} from the set of relative cluster-tilting subcategories to the set of support $\tau$-tilting pairs, see Section~\ref{ss:relation-to-support-tau-tilting-theory}.

\medskip
One particular interesting situation is when $\ct$ is a 2-Calabi--Yau triangulated category and $M$ is a cluster-tilting object.
In this case, an object of $\ct$ is relative cluster-tilting if and only if it is cluster-tilting.

\begin{theorem}
[{Theorem~\ref{thm:the-one-object-case-cto}}]
\label{thm:main-2}
Assume further that $k$ is a field, $\ct$ is a 2-Calabi--Yau triangulated $k$-category and $M$ is a cluster-tilting object.
\begin{itemize}
\item[(a)] The functor $\mathbb{P}$ is an equivalence if and only if $A$ is self-injective. 
\item[(b)] Assume that $A$ is self-injective. Then the bijection in Theorem~\ref{thm:main-1}(d) restricts to a bijection from (i) to (ii) of the following sets:
\begin{itemize}
\item[(i)] the set of isomorphism classes of basic cluster-tilting objects of $\ct$ with self-injective endomorphism algebras,
\item[(ii)] the set of isomorphism classes of basic 2-term tilting objects in $\ch^b(\proj A)$.
\end{itemize}
As a consequence, a basic 2-term silting object $N$ of $\ch^b(\proj A)$ is tilting if and only if its endomorphism algebra is self-injective.
\end{itemize}
\end{theorem}

As a consequence of Theorem~\ref{thm:main-2}(a), if in addition $A$ is self-injective, then the category $\ch^{[-1,0]}(\proj A)$ admits a triangle structure.
Conversely, if $A$ is any finite-dimensional algebra such that $\ch^{[-1,0]}(\proj A)$ admits a triangle structure, then $A$ is necessarily self-injective, see Appendix~\ref{a:triangle-structures-on-2-term-complexes} written by Osamu Iyama.

\smallskip
Main examples of $2$-Calabi--Yau triangulated categories with cluster-tilting objects are cluster categories of quivers with potential \cite{Amiot09,KellerYang11}. Let $(Q,W)$ be a quiver with potential and $\ct=\cc_{(Q,W)}$ be its cluster category. Then $\ct$ has a cluster-tilting object $M$ whose endomorphism algebra is the complete Jacobian algebra $J=\widehat{J}(Q,W)$ of $(Q,W)$. Applying Theorem~\ref{thm:main-2}(a) and Theorem~\ref{thm:main-1}(e) we obtain

\begin{theorem}[{Theorem~\ref{thm:compatibility-with-mutations}}]
\label{thm:main-3} 
Assume that $J$ is self-injective.
Let  $\underline{i}=(i_1,\ldots,i_s)$ be a sequence of vertices of $Q$ such that $i_{r}$ does not lie on a loop or 2-cycle in the   quiver of the iteratedly mutated quiver with potential $\mu_{i_r-1}\cdots\mu_{i_1}(Q,W)$ for all $1\leq r\leq s-1$ and let $\underline{\epsilon}=(\epsilon_1,\ldots,\epsilon_s)$ be a sequence of elements of $\{L,R\}$ such that the iteratedly mutated silting object $J_{\underline{i},\underline{\epsilon}}=\mu_{i_r}^{\epsilon_r}\cdots\mu_{i_1}^{\epsilon_1}(J)$ belongs to $\ch^{[-1,0]}(\proj J)$ for any $r\in\{1,\ldots,s\}$. Then 
\[
\End_{\ch^b(\proj J)}(J_{\underline{i},\underline{\epsilon}})\cong \widehat{J}(\mu_{i_s}\cdots\mu_{i_1}(Q,W)).
\]
\end{theorem}

Applying Theorem~\ref{thm:main-3} to the case when $\underline{i}$ consists of vertices of $Q$ on which there are no loops or 2-cycles and between which there are no arrows, we obtain \cite[Theorem 1.1(a)]{Mizuno15}. Actually the whole study in this paper was motivated by this theorem of Mizuno.

\medskip
During the preparation of this work, Yann Palu informed us of the work \cite{FGPPP23}, which establishes Theorem~\ref{thm:main-1}(a)(c) for non-positive dg categories in a relative version (\cite[Theorem B]{FGPPP23}) and the `if' part of Theorem~\ref{thm:main-2}(a) for cluster categories of quivers with potential (\cite[Corollary 6.16]{FGPPP23}), and the work \cite{ChenXF23}, which establishes Theorem~\ref{thm:main-1}(a)(c) in the general setting of algebraic hereditary extriangulated categories (\cite[Example 7]{ChenXF23}). These three projects were carried on simultaneously with different motivations. Also, Bernhard Keller pointed out to us that the functor $\mathbb{P}$ exists if $\ct$ has a morphic enhancement in the sense of \cite[Appendix C]{Krause20}.

\medskip
\noindent{\it Notation.}
Throughout, let $k$ be a commutative ring and we will denote by $\Sigma$ the shift functor of all triangulated categories.

For an additive $k$-category $\ca$, we denote by $\ch^b(\ca)$ the homotopy category of bounded complexes in $\ca$, and by $\ch^{[-1,0]}(\ca)$ its full subcategory consisting of complexes which are isomorphic to complexes concentrated in degrees $-1$ and $0$ (we call such complexes \emph{2-term complexes}). For an object $X$ of $\ca$, we denote by $\add(X)=\add_\ca(X)$ the full subcategory of $\ca$ consisting of direct summands of finite direct sums of copies of $X$. If $\ca$ is in addition triangulated, we denote by $\thick(X)=\thick_\ca(X)$ the smallest triangulated subcategory of $\ca$ which contains $X$ and which is closed under taking direct summands.

\section{The category of finitely presented objects}
\label{s:the-category-of-finitely-presented-objects}

In this section we recall some basics on the category finitely presented by a rigid object in a triangulated category, including the definition, the extriangle structure, relative cluster-tilting objects/subcategories and their mutations.

\smallskip
Throughout this section, let $\ct$ be an idempotent-complete triangulated $k$-category. A full subcategory $\cm$ of $\ct$ is said to be \emph{$d$-rigid} if $\Hom_\ct(M,\Sigma^p M')=0$ for any $M,M'\in\cm$ and any $1\leq p\leq d$, and it is said to be \emph{rigid} if it is $1$-rigid. It is said to be \emph{presilting} if it is $d$-rigid for any $d\in\mathbb{N}$, and \emph{silting} if in addition $\ct=\thick(\cm)$. An object $M$ of $\ct$ is said to be \emph{rigid} (respectively, \emph{presilting}, \emph{silting})  if $\add(M)$ is a rigid (respectively, presilting, silting) subcategory of $\ct$, equivalently, if $\Hom_\ct(M,\Sigma M)=0$ (respectively, $\Hom_\ct(M,\Sigma^p M)=0$ for any $p>0$, $\Hom_\ct(M,\Sigma^p M)=0$ for any $p>0$ and $\ct=\thick(M)$). 

In the sequel of this section we fix a rigid subcategory $\cm$ which is closed under taking finite direct sums and direct summands.

\subsection{The category}
\label{ss:the-category}

Put
\[
\fpr_\ct(\cm):=\{X\in\ct\mid \exists\text{ a triangle }\xymatrix@C=0.5pc{M^{-1}\ar[r] & M^0\ar[r] & X\ar[r] & \Sigma M^{-1}}\text{ with }M^{-1},M^0\in\cm\}.
\]
For example, $\fpr_{\ch^b(\cm)}(\cm)=\ch^{[-1,0]}(\cm)$. 
If it does not cause confusion, we will drop the subscript and simply write $\fpr(\cm)$. If $M$ is a rigid object of $\ct$, we write $\fpr(M)$ for $\fpr(\add(M))$.

\begin{remark}
In the literature $\fpr(\cm)$ is also denoted by $\cm*\Sigma\cm$. By \cite[Lemma 2.6(b)]{IyamaYang18}, $\fpr(\cm)$ is idempotent-complete. If $\cm$ is presilting, then $\fpr(\cm)$ is extension-closed but not closed under shifts. If $\cm$ is a cluster-tilting subcategory, then $\ct=\fpr(\cm)$, see \cite[Lemma 3.2]{KoenigZhu08}. If $\ct$ is a cluster tube and $\cm=\add(M)$ with $M$ a maximal rigid object, then $\fpr(\cm)$ is neither extension-closed nor closed under shifts, see \cite[Section 4.3]{Yang12} for a description of $\fpr(\cm)$ in this case.
\end{remark}

\subsection{The extriangle structure}

For $X,Y\in\fpr(\cm)$, let $[\Sigma\cm](X,\Sigma Y)$ denote the subset of morphisms $X\to \Sigma Y$ which factors through an object in $\Sigma\cm$. It is a remarkable idea of \cite{YangZhu19} to consider this space instead of $\Hom_\ct(X,\Sigma Y)$. 

Given two triangles
\[
\xymatrix@R=0.5pc{
\eta\colon M^{-1}\ar[r]^(0.55)a & M^0 \ar[r]^b & X\ar[r]^(0.4)c & \Sigma M^{-1},\\
\eta'\colon M'^{-1}\ar[r]^(0.55){a'} & M'^0 \ar[r]^{b'} & Y\ar[r]^(0.4){c'} & \Sigma M'^{-1} 
}
\]
with $M^{-1},M^0,M'^{-1},M'^0\in\cm$, we denote the 2-term complexes $(M^{-1}\stackrel{a}{\to}M^0)$ and $(M'^{-1}\stackrel{a'}{\to}M'^0)$ by $t(\eta)$ and $t(\eta')$, respectively. 
 Note that $\Hom_{\ch^b(\cm)}(t(\eta),\Sigma t(\eta'))$ consists of homotopy classes of morphisms $M_1\to M'_0$. For a morphism $f\colon M^{-1}\to M'^0$, we have a morphism $\Sigma b'\circ\Sigma f\circ c\in[\Sigma\cm](X,\Sigma Y)$:
\[
\xymatrix@R=2.5pc{
M^{-1}\ar[r]^{a} & M^0\ar[r]^(0.55){b} & X\ar[r]^(0.4){c} & \Sigma M^{-1}\ar[dll]|(0.3){\Sigma f}\\
\Sigma M'^{-1}\ar[r]^{-\Sigma a'} & \Sigma M'^0\ar[r]^(0.55){-\Sigma b'} & \Sigma Y\ar[r]^(0.4){-\Sigma c'} & \Sigma^2 M'^{-1}.
}
\]

\begin{lemma}
\label{lem:taking-presentation-preserves-extension}
The assignment $f\mapsto \Sigma b'\circ\Sigma f\circ c$ induces a bijection
\[
\Hom_{\ch^b(\cm)}(t(\eta),\Sigma t(\eta'))\longrightarrow [\Sigma\cm](X,\Sigma Y).
\]
\end{lemma}
\begin{proof}
We first show that the map is well-defined. It suffices to show that if $f$ is null-homotopic, then $\Sigma b'\circ\Sigma f\circ c=0$. Because $f$ is null-homotopic, there exist
morphisms $h^0\colon M^0\to M'^0$ and $h^{-1}\colon M^{-1}\to M'^{-1}$ such that $f=h^0\circ a-a'\circ h^{-1}$:
\[
\xymatrix{
& M^{-1}\ar@{.>}[dl]_{h^{-1}}\ar[r]^{a}\ar[d]^{f} & M^0\ar@{.>}[dl]^{h^0} \\
M'^{-1}\ar[r]^{-a'} & M'^0.
}
\]
Therefore
\begin{eqnarray*}
\Sigma b'\circ\Sigma f\circ c&=&\Sigma b'\circ (\Sigma h^0\circ\Sigma a-\Sigma a'\circ\Sigma h^{-1})\circ c\\
&=&\Sigma b'\circ \Sigma h^0\circ\Sigma a\circ c+\Sigma b'\circ \Sigma a'\circ\Sigma h^{-1}\circ c\\
&=&0.
\end{eqnarray*}

Next we show that the map is surjective. Let $g\in[\Sigma\cm](X,\Sigma Y)$. Then there exist an object $M\in\cm$ and two morphisms $g_1\colon X\to\Sigma M$ and $g_2\colon \Sigma M\to\Sigma Y$ such that $g=g_2\circ g_1$. Since $g_1\circ b=0$, there exists a morphism $f_1\colon M^{-1}\to M$ such that $g_1=\Sigma f_1\circ c$. Since $(-\Sigma c')\circ g_2=0$, there exists a morphism $f_2\colon M\to M'^0$ such that $g_2=-\Sigma b'\circ \Sigma f_2$.
\[
\xymatrix{
M^0\ar[r]^{b} & X\ar[r]^{c}\ar[d]^{g_1} &\Sigma M^{-1}\ar[r]^{\Sigma a}\ar@{.>}[ld]^{\Sigma f_1} & \Sigma M^0\\
& \Sigma M\ar@{.>}[ld]_{\Sigma f_2}\ar[d]^{g_2} &&&\\
\Sigma M'^0\ar[r]^{-\Sigma b'} & \Sigma Y\ar[r]^{-\Sigma c'} & \Sigma^2 M'^{-1}
}
\]
Therefore letting $f=-f_2\circ f_1$ we have $g=\Sigma b'\circ\Sigma f\circ c$.

Finally, we show that the map is injective. Suppose $\Sigma b'\circ\Sigma f\circ c=0$, \ie $b'\circ f\circ(-\Sigma^{-1}c)=0$. Then  there exists a morphism $h'\colon M^0\to Y$ such that $h'\circ a=b'\circ f$. But $c'\circ h'=0$, so there exists $h^0\colon M^0\to M'^0$ such that $h'=b'\circ h^0$. Now $b'\circ (f-h^0\circ a)=0$, so there exists $h^{-1}\colon M^{-1}\to M'^{-1}$ such that $f-h^0\circ a=a'\circ h^{-1}$, hence $f$ is null-homotopic.
\[
\xymatrix@C=4pc{
\Sigma^{-1}X\ar[r]^{-\Sigma^{-1}c} & M^{-1}\ar[r]^{a}\ar@{.>}[ld]|{h^{-1}}\ar[d]^f & M^0\ar@{.>}[dl]|{h^0} \ar@{.>}[d]^{h'}\\
M'^{-1}\ar[r]^{a'} & M'^0\ar[r]^{b'} & Y\ar[r]^{c'} & \Sigma M'^{-1}
} 
\]
\end{proof}

The following horseshoe-lemma type result is a special case of the first statement of \cite[Lemma 4.57]{PadrolPaluPilaudPlamondon23}. It can be directly proved by using Lemma~\ref{lem:taking-presentation-preserves-extension} and the octahedral axiom.

%For completeness we include a proof. 

\begin{lemma}
\label{lem:the-extriangle-structure}
$\fpr(\cm)$ is $[\Sigma\cm]$-extension-closed, that is, if $\xymatrix@=0.7pc{Y\ar[r] & Z\ar[r] & X\ar[r]^w & \Sigma Y}$ is a triangle in $\ct$ with $w\in[\Sigma\cm](X,\Sigma Y)$, then $Z\in\fpr(M)$.
\end{lemma}

A triangle $\xymatrix@=0.7pc{Y\ar[r] & Z\ar[r] & X\ar[r]^w & \Sigma Y}$ in $\ct$ with $X,Y,Z\in\fpr(\cm)$ is called an \emph{extriangle} if $w\in[\Sigma\cm](X,\Sigma Y)$. We call the set of extriangles \emph{the extriangle structure} on $\fpr(\cm)$. Indeed, with these extriangles $\fpr(\cm)$ is an extriangulated category in the sense of \cite{NakaokaPalu19}, due to \cite[Lemma 4.57]{PadrolPaluPilaudPlamondon23}. In fact, it is a reduced $0$-Auslander extriangulated category in the sense of \cite[Section 3.2]{GorskyNakaokaPalu23} with $\cm$ being the full subcategory of projective objects. We will use this fact for Lemmas~\ref{lem:number-of-summands-of-cto} and~\ref{lem:mutation-of-relative-cluster-tilting}.

\subsection{Relative cluster-tilting subcategories}

We follow \cite[Definition 2.3]{ZhouZhu20} but we adopt the terminologies in \cite{YangZhu19}. 
Let $\cn$ be a subcategory of $\fpr(\cm)$ which is closed under taking finite direct sums and direct summands. It is said to be
\begin{itemize}
\item[(a)] \emph{relative rigid} if $[\Sigma\cm](X,\Sigma Y)=0$ for any $X,Y\in\cn$;
\item[(b)]  \emph{maximal relative rigid} if it is relative rigid and if $X\in\fpr(\cm)$ satisfies that $[\Sigma \cm](X,\Sigma X)=0$ and that $[\Sigma \cm](X,\Sigma N)=0=[\Sigma \cm](N,\Sigma X)$ for all $N\in\cn$, then $X\in\cn$;
\item[(c)] \emph{generating} if $\cm\subseteq \Sigma^{-1}\cn*\cn$;
\item[(d)]  \emph{weakly relative cluster-tilting} if it is relative rigid and generating and 
\[
\cn=\{X\in\fpr(\cm)\mid [\Sigma \cm](X,\Sigma N)=0=[\Sigma \cm](N,\Sigma X)\};
\]
\item[(e)] \emph{relative cluster-tilting} if it is relative rigid, contravariantly finite and 
\[
\cn=\{X\in\fpr(\cm)\mid [\Sigma \cm](X,\Sigma N)=0=[\Sigma \cm](N,\Sigma X)\}.
\]
\end{itemize}
Clearly, if $\cn$ is rigid, then it is relative rigid. Note that $\cn$ is tilting in the sense of \cite[Section 4.1]{GorskyNakaokaPalu23} if and only if it is relative rigid and generating.

\begin{lemma}
\label{lem:comparing-two-rigidity}
Assume that $\cm$ is $2$-rigid. Then $\Hom(X,\Sigma Y)=[\Sigma\cm](X,\Sigma Y)$ for any $X,Y\in\fpr(\cm)$. As a consequence, $\cn$ is relative rigid if and only if $\cn$ is rigid. 
\end{lemma}
\begin{proof}
Take $f\in\Hom_\ct(X,\Sigma Y)$ and assume that the two rows in the following diagram are triangles with $M^0,M^{-1},M'^0,M'^{-1}\in\cm$:
\[
\xymatrix{
M^{-1}\ar[r] & M^0\ar[r]^{b} & X\ar[r]^{c}\ar[d]^f & \Sigma M^{-1}\\
\Sigma M'^{-1}\ar[r] & \Sigma M'^0\ar[r]^{-\Sigma b'} & \Sigma Y\ar[r]^{-\Sigma c'} & \Sigma^2 M'^{-1}.
}
\]
Since $\cm$ is 2-rigid, we have $-\Sigma c'\circ f\circ b=0$. So $f\circ b$ factors through $\Sigma M'^0$, and hence $f\circ b=0$. Thus $f$ factors through $\Sigma M^{-1}$, and hence belongs to $[\Sigma \cm](X,\Sigma Y)$. 
\end{proof}

An object $N\in\fpr(\cm)$ is said to be \emph{relative rigid}, \emph{maximal relative rigid}, \emph{generating}, \emph{weakly relative cluster-tilting} and \emph{relative cluster-tilting} if $\add(N)$ is relative rigid, maximal relative rigid, generating, weakly relative cluster-tilting and relative cluster-tilting, respectively.

\begin{lemma}
\label{lem:number-of-summands-of-cto}
Assume that $\cm=\add(M)$ for a rigid object $M$ and that $k$ is a field and $\fpr(M)$ is Hom-finite over $k$. The following conditions are equivalent for a relative rigid object $N$ of $\fpr(M)$:
\begin{itemize}
\item[(i)] $N$ is relative cluster-tilting,
\item[(ii)] $N$ is maximal relative rigid,
\item[(iii)] $N$ is generating,
\item[(iv)] $|M|=|N|$.
\end{itemize}
\end{lemma}
\begin{proof}
The equivalence (i)$\Leftrightarrow$(ii) is obtained as a consequence of \cite[Theorem 3.1]{ZhouZhu20}. The equivalences (ii)$\Leftrightarrow$(iii)$\Leftrightarrow$(iv) are obtained by applying \cite[Theorem 4.3]{GorskyNakaokaPalu23}.
\end{proof}

\subsection{Mutation}
Assume that $k$ is a field and $\fpr(\cm)$ is Hom-finite over $k$. Assume further that $\cm=\add(M)$ for a rigid object $M$.

Let $N$ be a basic relative cluster-tilting object of $\fpr(M)$ and write $N=N_1\oplus\ldots\oplus N_n$ with $N_1,\ldots,N_n$ indecomposable. Fix $i=1,\ldots,n$. The following result is \cite[Corollary 4.28]{GorskyNakaokaPalu23} applied to $\fpr(M)$ (see also \cite[Section 2]{HeZhouZhu21}).

\begin{lemma} 
\label{lem:mutation-of-relative-cluster-tilting}
There exists an indecomposable object $N_i^*$, unique up to isomorphism, such that $N_i^*$ is not isomorphic to $N_i$ and $\bigoplus_{j\neq i}N_j\oplus N_i^*$ is a relative cluster-tilting object of $\fpr(M)$. Moreover, there is 
an extriangle
\begin{align}
\xymatrix{
N_i\ar[r] & E \ar[r] & N_i^*\ar[r]&\Sigma N_i 
},\label{triangle:left-mutation}
\end{align}
or an extriangle
\begin{align}
\xymatrix{
N_i^*\ar[r] & E'\ar[r] & N_i\ar[r] &\Sigma N_i^*
},\label{triangle:right-mutation}
\end{align}
but not both at the same time. In the former case we write $\mu_i^L(N)=\bigoplus_{j\neq i}N_j\oplus N_i^*$ and in the latter case we write $\mu_i^R(N)=\bigoplus_{j\neq i}N_j\oplus N_i^*$.
\end{lemma}

\subsection{Special case: $\ct=\ch^b(\cm)$}

Assume that $\ct=\ch^b(\cm)$. In this case, $\fpr(\cm)=\ch^{[-1,0]}(\cm)$, and a triangle of $\ch^b(\cm)$ with all three terms in $\ch^{[-1,0]}(\cm)$ is an extriangle in $\ch^{[-1,0]}(\cm)$.
A full subcategory (respectively, an object) of $\ch^b(\cm)$ is said to be \emph{2-term} if it is contained in $\fpr(\cm)=\ch^{[-1,0]}(\cm)$.

\begin{lemma}[{\cite[Theorem 5.4]{ZhouZhu20} and \cite[Theorem 3.4]{FuGengLiu19}}]
\label{lem:rigid-vs-presilting-in-homotopy-category}
Let $\cn$ be a 2-term subcategory of $\ch^b(\cm)$ which is closed under taking finite direct sums and direct summands.
Consider the following conditions:
\begin{itemize}
\item[(i)] $\cn$ is relative rigid,
\item[(ii)] $\cn$ is rigid,
\item[(iii)] $\cn$ is presilting,
\item[(iv)] $\cn$ is weakly relative cluster-tilting,
\item[(v)] $\cn$ is silting.
\end{itemize}
Then (i)$\Leftrightarrow$(ii)$\Leftrightarrow$(iii) and (iv)$\Leftrightarrow$(v).
\end{lemma}

(i)$\Leftrightarrow$(ii) follows by Lemma~\ref{lem:comparing-two-rigidity}. (ii)$\Leftrightarrow$(iii) holds because $\cn$ is 2-term. According to \cite[Corollary 5.10]{SaorinZvonareva22}, if $\ct$ is small, then any silting subcategory in $\ct$ is weakly precovering and weakly preenveloping in the sense of \cite[Definition 5.8]{SaorinZvonareva22}.

Assume further that $k$ is a field and $\cm=\add(M)$ is Hom-finite over $k$. Then in this case the mutations introduced in the preceding subsection are exactly irreducible silting mutations in the sense of Aihara--Iyama~\cite{AiharaIyama12}.

\section{From finitely presented objects to 2-term complexes}
\label{s:from-finitely-presented-objects-to-two-term-complexes}

In this section we construct the functor $\mathbb{P}\colon\fpr(\cm)\to\ch^{[-1,0]}(\cm)$ and study its basic properties, especially its behaviour on morphisms and extensions as well as its compatibility with relative cluster-tilting subcategories.

\smallskip
Let $\ct$ be an idempotent-complete algebraic triangulated $k$-category and $\cm$ a rigid subcategory of $\ct$ which is closed under taking finite direct sums and direct summands.

\subsection{Lifting morphisms to presentations}
\label{ss:lifting-morphisms}

Assume that $\ct=\underline{\cf}$, where $\cf$ is a Frobenius $k$-category.
Let $\widetilde{\cm}$ be the full subcategory of $\cf$ which has the same objects as $\cm$, \ie $\underline{(\widetilde{\cm})}=\cm$. In particular, $\widetilde{\cm}$ contains all projective-injective objects. Let $X,Y$ be objects of $\cf$ with two conflations
\[
\xymatrix@R=0.5pc{
M^{-1}\ar[r]^a & M^0\ar[r]^b & X,\\
M'^{-1}\ar[r]^{a'} & M'^0\ar[r]^{b'} & Y,
}
\]
where $M^{-1},M^0,M'^{-1},M'^0\in\widetilde{\cm}$.

\begin{lemma}
\label{lem:lifting-morphisms}
For any morphism $f\colon X\to Y$ in $\cf$, there is a morphism of conflations
\[
\xymatrix{
M^{-1}\ar[r]^a \ar[d]^{f^{-1}}& M^0\ar[r]^b \ar[d]^{f^0} & X\ar[d]^f,\\
M'^{-1}\ar[r]^{a'} & M'^0\ar[r]^{b'} & Y.
}
\]
If $f$ factors through a projective-injective object $P$, then
there is a morphism $h\colon M^0\to M'^{-1}$ such that $f^{-1}=h\circ a$ and $f^0-a'\circ h$ factors through $P$.
\end{lemma}
\begin{proof}
As $\Ext_\cf^1(M^0,M'^{-1})\cong\Hom_\ct(M^0,\Sigma M'^{-1})=0$, it follows from the Hom-Ext long exact sequence that there exists a morphism $f^0\colon M^0\to M'^0$ such that the right square is commutative. The existence of $f^{-1}$ then follows from the universal property of the kernel.

Next assume that $f$ factors through a projective-injective object $P$, \ie there are morphisms $u\colon X\to P$ and $v\colon P\to Y$ such that $f=v\circ u$ in $\cf$. Since $P$ is projective, there exists $w\colon P\to M'^0$ such that $v=b'\circ w$ in $\cf$.
\[
\xymatrix@=1.5pc{
M^{-1}\ar[rrr]^(0.4){a}\ar[ddd]^{f^{-1}} &&& M^0\ar[rrr]^(0.6){b}\ar[ddd]^{f^0} \ar@{.>}[lllddd]|h&&& X\ar[ddd]^{f}\ar@{-->}[ldd]_{u}\\
\\
&&&&&P\ar@{-->}[lld]_{w}\ar@{-->}[rd]^{v}\\
M'^{-1}\ar[rrr]^(0.4){a'} &&& M'^0\ar[rrr]^(0.6){b'} &&& Y
}
\]
Now $b'\circ (f^0-w\circ u\circ b)=b'\circ f^0-f\circ b=0$, so there exists a morphism $h\colon M^0\to M'^{-1}$ such that $f^0-w\circ u\circ b=a'\circ h$ in $\cf$. Therefore $f^0-a'\circ h$ factors through $P$. Moreover, $a'\circ(h\circ a-f^{-1})=f\circ a-a'\circ f^0=0$, so $f^{-1}=h\circ a$ in $\cf$, since $a'$ is a monomorphism.
\end{proof}

\subsection{The functor}
\label{ss:the-functor}

For any $X\in\fpr(\cm)$, fix a triangle in $\ct$
\begin{align}\label{eq:presentation}
\xymatrix{
M^{-1}_X\ar[r]^{a_X} & M^0_X\ar[r]^(0.55){b_X} & X\ar[r]^(0.4){c_X} & \Sigma M^{-1}_X
}
\end{align}
with $M^{-1}_X,M^0_X\in\cm$. 
For a morphism $f\colon X\to Y$, we call a pair $(f^{-1},f^0)$ of morphisms $f^{-1}\colon M^{-1}_X\to M^{-1}_Y$ and $f^0\colon M^0_X\to M^0_Y$ a \emph{presentation} of $f$ if there is a morphism of triangles
\begin{align}\label{eq:presentation-of-morphism}
\xymatrix{
M^{-1}_X\ar[r]^{a_X}\ar[d]^{f^{-1}} & M^0_X\ar[r]^(0.55){b_X}\ar[d]^{f^0} & X\ar[r]^(0.4){c_X}\ar[d]^{f} & \Sigma M^{-1}_X\ar[d]^{\Sigma f^{-1}}\\
M^{-1}_Y\ar[r]^{a_Y} & M^0_Y\ar[r]^(0.55){b_Y} & X\ar[r]^(0.4){c_Y} & \Sigma M^{-1}_Y.
}
\end{align}
It is said to be \emph{liftable} if, roughly speaking, the above morphism of triangles lifts to a morphism of conflations in the underlying Frobenius category. The rigorous definition will be given in the beginning of the proof below.

\begin{proposition}
\label{prop:from-fpr-to-2term}
For an object $X\in \fpr(\cm)$, define $\mathbb{P}(X)$ as the 2-term complex $(M^{-1}_X\stackrel{a_X}{\longrightarrow} M^0_X)$; for a morphism $f\colon X\to Y$ of $\fpr(\cm)$, define $\mathbb{P}(f)$ as the homotopy class of a liftable presentation of $f$. Then $\mathbb{P}\colon \fpr(\cm)\to \ch^{[-1,0]}(\cm)$ is a dense $k$-linear functor. 
\end{proposition}
\begin{proof} Let us assume that $\ct=\underline{\cf}$, where $\cf$ is a Frobenius $k$-category, as in Section~\ref{ss:lifting-morphisms}.
Denote by $\pi\colon \cf\to\ct$ the projection functor. We will not distinguish in notation a morphism in $\cf$ and its image in $\ct$.

When defining the triangle structure on $\ct$, we have fixed for any object $X$ of $\cf$ a conflation
\[
\xymatrix{
X\ar[r]^(0.45){i(X)} & I(X)\ar[r]^{p(X)} & \Sigma X
}
\]
with $I(X)$ being projective-injective, and for any morphism $f\colon X\to Y$ of $\cf$ a morphism of conflations
\[
\xymatrix{
X\ar[r]^(0.45){i(X)}\ar[d]^{f} & I(X)\ar[r]^{p(X)}\ar[d]^{I(f)} & \Sigma X\ar[d]^{\Sigma f}\\
Y\ar[r]^(0.45){i(Y)} & I(Y)\ar[r]^{p(Y)} & \Sigma Y
}
\]

Now let $X\in\fpr(\cm)$. Consider the inflation $\widetilde{a}_X={a_X\choose {i(M^{-1}_X)}}\colon M^{-1}_X\longrightarrow M^0_X\oplus I(M^{-1}_X)$ and complete it to a conflation in $\cf$ 
\[
\xymatrix{
M^{-1}_X\ar[r]^(0.35){\widetilde{a}_X} & M^0_X\oplus I(M^{-1}_X)\ar[r]^(0.7){\widetilde{b}_X} & \widetilde{X},
}
\]
which produces a triangle in $\ct$
\[
\xymatrix{
M^{-1}_X\ar[r]^(0.35){\widetilde{a}_X} & M^0_X\oplus I(M^{-1}_X)\ar[r]^(0.7){\widetilde{b}_X} & \widetilde{X}\ar[r]^(0.4){\widetilde{c}_X} & \Sigma M^{-1}_X.
}
\]
Then in $\ct$ there is an isomorphism of triangles
\begin{align}
\label{cd:eta}
\xymatrix{
M^{-1}_X\ar[r]^(0.35){\widetilde{a}_X}\ar@{=}[d] & M^0_X\oplus I(M^{-1}_X)\ar[r]^(0.7){\widetilde{b}_X} \ar[d]^{(\id,0)}& \widetilde{X} \ar[r]^(0.4){\widetilde{c}_X} \ar[d]^{\eta_X}& \Sigma M^{-1}_X\ar@{=}[d]\\
M^{-1}_X\ar[r]^{a_X} & M^0_X\ar[r]^(0.55){b_X} & X\ar[r]^(0.4){c_X} & \Sigma M^{-1}_X.
}
\end{align}
We fix such an isomorphism $\eta_X\colon \widetilde{X}\to X$ and let $\epsilon_X\colon X\to \widetilde{X}$ be its inverse in $\ct$.

Let $f\colon X\to Y$ be a morphism in $\fpr(\cm)$. A presentation $(f^{-1},f^0)$ of $f$ is said to be \emph{liftable} if there is a morphism of conflations in $\cf$
\begin{align}\label{cd:liftable-presentation}
\xymatrix{
M^{-1}_X\ar[r]^(0.35){\widetilde{a}_X}\ar[d]^{f^{-1}} & M^0_X\oplus I(M^{-1}_X)\ar[r]^(0.7){\widetilde{b}_X}\ar[d]^{\widetilde{f}^0} & \widetilde{X}\ar[d]^{\widetilde{f}}\\
M^{-1}_Y\ar[r]^(0.35){\widetilde{a}_Y} & M^0_Y\oplus I(M^{-1}_Y)\ar[r]^(0.7){\widetilde{b}_Y} & \widetilde{Y}
}
\end{align}
such that the following diagram is commutative in $\ct$
\[
\xymatrix@=1pc{
&M^{-1}_X\ar[rrr]|{a_X}\ar@{.>}[ddd]|(0.65){f^{-1}} &&& M^0_X\ar[rrr]|{b_X}\ar@{.>}[ddd]|(0.65){f^0} &&& X\ar[rrr]|{c_X} \ar@{.>}[ddd]|(0.65){f}&&& \Sigma M^{-1}_X\ar[ddd]|(0.65){\Sigma f^{-1}}\\
M^{-1}_X\ar[rrr]|{\widetilde{a}_X}\ar[ddd]|(0.35){f^{-1}}\ar@{=}[ur] &&& M^0_X\oplus I(M^{-1}_X)\ar[rrr]|(0.7){\widetilde{b}_X}\ar[ddd]|(0.35){\widetilde{f}^0} \ar[ur]|{(\id,0)}&&& \widetilde{X}\ar[ddd]|(0.35){\widetilde{f}}\ar[ur]|{\eta_X} \ar[rrr]|(0.55){\widetilde{c}_X} &&& \Sigma M^{-1}_X\ar@{=}[ur]\ar[ddd]|(0.35){\Sigma f^{-1}}\\
\\
&M^{-1}_Y\ar@{.>}[rrr]|(0.4){a_Y} &&& M^0_Y\ar@{.>}[rrr]|(0.4){b_Y} &&& Y\ar@{.>}[rrr]|(0.4){c_Y} &&& \Sigma M^{-1}_Y\\
M^{-1}_Y\ar[rrr]|(0.4){\widetilde{a}_Y}\ar@<.2ex>@{.}[ur]\ar@<-.2ex>@{.}[ur] &&& M^0_Y\oplus I(M^{-1}_Y)\ar[rrr]|(0.6){\widetilde{b}_Y}\ar@{.>}[ur]|{(\id,0)} &&& \widetilde{Y}\ar@{.>}[ur]|{\eta_Y}\ar[rrr]|{\widetilde{c}_Y} &&& \Sigma M^{-1}_Y\ar@{=}[ur]
}
\]

(1) \emph{Every morphism in $\fpr(\cm)$ admits a liftable presentation.} Let $f\colon X\to Y$ be a morphism in $\fpr(\cm)$. Let $\widetilde{f}\colon \widetilde{X}\to\widetilde{Y}$ be a morphism in $\cf$ such that $f=\eta_Y\circ\widetilde{f}\circ\epsilon_X$ in $\ct$.
By the first statement of Lemma~\ref{lem:lifting-morphisms} there is a morphism of conflations
\begin{align*}
\xymatrix{
M^{-1}_X\ar[r]^(0.35){\widetilde{a}_X}\ar[d]^{f^{-1}} & M^0_X\oplus I(M^{-1}_X)\ar[r]^(0.7){\widetilde{b}_X}\ar[d]^{\widetilde{f}^0} & \widetilde{X}\ar[d]^{\widetilde{f}}\\
M^{-1}_Y\ar[r]^(0.35){\widetilde{a}_Y} & M^0_Y\oplus I(M^{-1}_Y)\ar[r]^(0.7){\widetilde{b}_Y} & \widetilde{Y}
}
\end{align*}
Then $(f^{-1},(\id,0)\circ\widetilde{f}^0\circ{\id\choose 0})$ is a liftable presentation of $f$.

(2)  \emph{A liftable presentation of the zero morphism is null-homotopic.} Let $f=0\colon X\to Y$ and let $(f^{-1},f^0)$ be a liftable presentation of $f$. Consider the diagram \eqref{cd:liftable-presentation}. The morphism $\widetilde{f}$ factors through a projective-injective object $P$, so by the second statement of Lemma~\ref{lem:lifting-morphisms} there exists a morphism $h\colon M^0_X\oplus I(M^{-1}_X)\to M^{-1}_Y$ in $\cf$ such that $\widetilde{f}^0-\widetilde{a}_Y\circ h$ factors through $P$ and $f^{-1}=h\circ\widetilde{a}_X$ in $\cf$. Therefore in $\ct$ we have $f^0=a_Y\circ h$ and $f^{-1}=h\circ a_X$, showing that $(f^{-1},f^0)$ is null-homotopic. 

(3) \emph{Let $(f^{-1},f^0)$ be a liftable presentation of $f$ and $\lambda\in k$. Then $(\lambda f^{-1},\lambda f^0)$ is a liftable presentation of $\lambda f$.} This is clear.

(4) \emph{Let $(f^{-1},f^0)$ be a liftable presentation of $f\colon X\to Y$ and $(g^{-1},g^0)$ be a liftable presentation of $g\colon X\to Y$. Then $(f^{-1}+g^{-1},f^0+g^0)$ is a liftable presentation of $f+g\colon X\to Y$.} This is clear.

(5) \emph{$\mathbb{P}$ is well-defined on morphisms.} Let $f\colon X\to Y$ be a morphism in $\fpr(\cm)$ and let $(f^{-1},f^0)$ and $(f'^{-1},f'^0)$ be two liftable presentations of $f$. Then $(f^{-1}-f'^{-1},f^0-f'^0)$ is a liftable presentation of the zero morphism by (3) and (4) , so by (2) it is null-homotopic, \ie $(f^{-1},f^0)$ and $(f'^{-1},f'^0)$ are in the same homotopy class.

(6) \emph{Let $(f^{-1},f^0)$ be a liftable presentation of $f\colon X\to Y$ and $(g^{-1},g^0)$ be a liftable presentation of $g\colon Y\to Z$. Then $(g^{-1}\circ f^{-1},g^0\circ f^0)$ is a liftable presentation of $g\circ f\colon X\to Z$.} This is clear.

(7) \emph{$\mathbb{P}$ is a $k$-linear functor.} This follows from (3) (4) and (5) together with the fact that $\mathbb{P}(\id_X)=\id_{\mathbb{P}(X)}$.

(8) \emph{$\mathbb{P}$ is dense.} For a complex $Y=(M^{-1}\stackrel{a}{\to}M^0)$ in $\ch^{[-1,0]}(\cm)$, complete the inflation $ \widetilde{a}={a\choose {i(M^{-1})}}\colon M^{-1}\longrightarrow M^0\oplus I(M^{-1})$ to a conflation in $\cf$ 
\[
\xymatrix{
M^{-1}\ar[r]^(0.35){\widetilde{a}} & M^0\oplus I(M^{-1})\ar[r]^(0.7){\widetilde{b}} & X.
}
\]
By Lemma~\ref{lem:lifting-morphisms}, there are morphisms of conflations
\[
\xymatrix{
M^{-1}\ar[r]^(0.35){\widetilde{a}}\ar[d]^{f^{-1}} & M^0\oplus I(M^{-1})\ar[r]^(0.7){\widetilde{b}} \ar[d]^{\widetilde{f}^0}& X\ar[d]^{\epsilon_X}\\
M^{-1}_X\ar[r]^(0.35){\widetilde{a}_X}\ar[d]^{g^{-1}} & M^0_X\oplus I(M^{-1}_X)\ar[r]^(0.7){\widetilde{b}_X} \ar[d]^{\widetilde{g}_0} & \widetilde{X}\ar[d]^{\eta_X}\\
M^{-1}\ar[r]^(0.35){\widetilde{a}} & M^0\oplus I(M^{-1})\ar[r]^(0.7){\widetilde{b}} & X,
}
\]
and thus a morphism of conflations
\[
\xymatrix{
M^{-1}\ar[r]^(0.35){\widetilde{a}}\ar[d]^{g^{-1}f^{-1}} & M^0\oplus I(M^{-1})\ar[r]^(0.7){\widetilde{b}} \ar[d]^{\widetilde{g}^0\widetilde{f}^0}& X\ar[d]^{\eta_X\epsilon_X}\\
M^{-1}\ar[r]^(0.35){\widetilde{a}} & M^0\oplus I(M^{-1})\ar[r]^(0.7){\widetilde{b}} & X.
}
\]
Since $\eta_X\epsilon_X=\id$ in $\ct$, \ie $\id-\eta_X\epsilon_X$ factors in $\cf$ through a projective-injective object, it follows by the second statement of Lemma~\ref{lem:lifting-morphisms} that there exists a morphism $h\colon M^0\oplus I(M^{-1})\to M^{-1}$
such that $\id-g^{-1}f^{-1}=h\widetilde{a}$ and $\id-\widetilde{g}^0\widetilde{f}^0=\widetilde{a}h$ in $\ct$. Put $f^0=(\id,0)\widetilde{f}^0{\id\choose 0}\colon M^0\to M^0_X$ and $g^0=(\id,0)\widetilde{g}^0{\id\choose 0}\colon M^0_X\to M^0$. Then $(g^{-1}f^{-1},g^0f^0)$ is homotopy equivalent to $\id$. Similarly,  $(f^{-1}g^{-1},f^0g^0)$ is homotopy equivalent to $\id$. It follows that $Y\cong\mathbb{P}(X)$ in $\ch^{[-1,0]}(\cm)$.
\end{proof}

\begin{remark}
\label{rem:independence-of-the-choice-of-presentation}
Up to isomorphism the functor $\mathbb{P}$ depends neither on the choice of the presentation \eqref{eq:presentation} nor on the choice of the isomorphism $\eta_X\colon \widetilde{X}\to X$. Therefore in practice we often take the following presentations for $M\in\cm$
\[
\xymatrix@R=0.5pc{
0\ar[r] & M\ar[r]^{\id} & M\ar[r] & 0\\
M\ar[r] & 0\ar[r] & \Sigma M\ar[r]^{\id} & \Sigma M,
}
\]
and if necessary we can take minimal presentations for all objects in $\fpr(\cm)$.

Let us prove the above statement. Fix another presentation 
\begin{align*}
\xymatrix{
M'^{-1}_X\ar[r]^{a'_X} & M'^0_X\ar[r]^(0.55){b'_X} & X\ar[r]^(0.4){c'_X} & \Sigma M'^{-1}_X,
}
\end{align*}
which produces a conflation in $\cf$ 
\[
\xymatrix{
M'^{-1}_X\ar[r]^(0.35){\widetilde{a'}_X} & M'^0_X\oplus I(M'^{-1}_X)\ar[r]^(0.7){\widetilde{b'}_X} & \widetilde{X}'
}
\]
and a triangle in $\ct$
\[
\xymatrix{
M'^{-1}_X\ar[r]^(0.35){\widetilde{a'}_X} & M'^0_X\oplus I(M'^{-1}_X)\ar[r]^(0.7){\widetilde{b'}_X} & \widetilde{X}'\ar[r]^(0.4){\widetilde{c'}_X} & \Sigma M'^{-1}_X,
}
\]
and fix an isomorphism $\eta'_X\colon\widetilde{X}'\to X$ in $\ct$ to construct another functor $\mathbb{P}'\colon \fpr(\cm)\to \ch^{[-1,0]}(\cm)$.
Consider the morphism $\epsilon_X\circ \eta'_X\colon \widetilde{X}'\to \widetilde{X}$ in $\cf$. There is a morphism of conflations in $\cf$ by Lemma~\ref{lem:lifting-morphisms}
\[
\xymatrix{
M'^{-1}_X\ar[r]^(0.35){\widetilde{a'}_X}\ar[d]^{\varphi^{-1}_X} & M'^0_X\oplus I(M'^{-1}_X)\ar[r]^(0.7){\widetilde{b'}_X}\ar[d]^{\widetilde{\varphi}^0_X} & \widetilde{X}'\ar[d]^{\epsilon_X\circ \eta'_X}\\
M^{-1}_X\ar[r]^(0.35){\widetilde{a}_X} & M^0_X\oplus I(M^{-1}_X)\ar[r]^(0.7){\widetilde{b}_X} & \widetilde{X}.
}
\]
Put $\varphi^0_X=(\id,0)\circ \widetilde{\varphi}^0_X\circ {\id\choose 0}$. Then there is a morphism of triangles
\begin{align}
\label{cd:phi}
\xymatrix{
M'^{-1}_X\ar[r]^(0.45){a'_X}\ar[d]^{\varphi^{-1}_X} & M'^0_X\ar[r]^(0.55){b'_X}\ar[d]^{\varphi^0_X} & X\ar@{=}[d] \ar[r]^(0.45){c'_X} &\Sigma M'^{-1}_X\ar[d]^{\Sigma\varphi^{-1}_X}\\
M^{-1}_X\ar[r]^(0.45){a_X} & M^0_X\ar[r]^(0.55){b_X} & X\ar[r]^(0.45){c_X} & \Sigma M^{-1}_X.
}
\end{align}
Define $\varphi_X=(\varphi^{-1}_X,\varphi^0_X)\colon\mathbb{P}'(X)\to\mathbb{P}(X)$. We claim that $\varphi=(\varphi_X)_X\colon\mathbb{P}'\to\mathbb{P}$ is an isomorphism.

(1) \emph{$\varphi_X$ is functorial in $X\in\fpr(\cm)$.} Let $f\colon X\to Y$ be a morphism in $\fpr(\cm)$. Then there is a diagram in $\cf$ which is commutative except the three vertical squares 
\[
\xymatrix@=1pc{
&M'^{-1}_X\ar[rrr]|(0.35){\widetilde{a'}_X}\ar[dl]|{\varphi^{-1}_X}\ar@{.>}[ddd]|(0.65){f'^{-1}} &&& M'^0_X\oplus I(M'^{-1}_X)\ar[rrr]|(0.6){\widetilde{b'}_X}\ar[dl]|{\widetilde{\varphi^0_X}}\ar@{.>}[ddd]|(0.65){\widetilde{f}'^0} &&& \widetilde{X}'\ar[dl]|{\epsilon_X\circ \eta'_X}\ar[ddd]|(0.65){\widetilde{f}'}\\
M^{-1}_X\ar[rrr]|{\widetilde{a}_X}\ar[ddd]|(0.35){f^{-1}} &&& M^0_X\oplus I(M^{-1}_X)\ar[rrr]|(0.7){\widetilde{b}_X} \ar[ddd]|(0.35){\widetilde{f}^0}&&& \widetilde{X}\ar[ddd]|(0.35){\widetilde{f}}\\
\\
&M'^{-1}_Y\ar@{.>}[rrr]|(0.35){\widetilde{a'}_Y}\ar@{.>}[dl]|{\varphi^{-1}_Y} &&& M'^0_Y\oplus I(M'^{-1}_Y)\ar@{.>}[rrr]|(0.6){\widetilde{b'}_Y}\ar@{.>}[dl]|{\widetilde{\varphi}^0_Y} &&& \widetilde{Y}'\ar[dl]|{\epsilon_Y\circ \eta'_Y}\\
M^{-1}_Y\ar[rrr]|{\widetilde{a}_Y} &&& M^0_Y\oplus I(M^{-1}_Y)\ar[rrr]|(0.7){\widetilde{b}_Y} &&& \widetilde{Y}
}
\]
The right vertical square is commutative in $\ct$, so by Lemma~\ref{lem:lifting-morphisms} there exists a morphism $h\colon M'^0_X\oplus I(M'^{-1}_X)\to M^{-1}_Y$ such that $\widetilde{f}^0\circ\widetilde{\varphi}^0_X-\widetilde{\varphi}^0_Y\circ\widetilde{f}'^0=\widetilde{a}_Y\circ h$ and $f^{-1}\circ\varphi^{-1}_X-\varphi^{-1}_Y\circ f'^{-1}=h\circ \widetilde{a'}_X$ in $\ct$. Therefore the following diagram is commutative in $\ch^{[-1,0]}(\cm)$
\[
\xymatrix@=1pc{
&M'^{-1}_X\ar[rrr]|(0.35){a'_X}\ar[dl]|{\varphi^{-1}_X}\ar@{.>}[ddd]|(0.65){f'^{-1}} &&& M'^0_X\ar[dl]|{\varphi^0_X}\ar[ddd]|(0.65){f'^0}\\
M^{-1}_X\ar[rrr]|{a_X}\ar[ddd]|(0.35){f^{-1}} &&& M^0_X\ar[ddd]|(0.35){f^0}\\
\\
&M'^{-1}_Y\ar@{.>}[rrr]|(0.35){a'_Y}\ar@{.>}[dl]|{\varphi^{-1}_Y} &&& M'^0_Y\ar[dl]|{\varphi^0_Y} \\
M^{-1}_Y\ar[rrr]|{{a}_Y} &&& M^0_Y
}
\]
showing the functoriality of $\varphi$.

(2) \emph{$\varphi_X$ is an isomorphism for any $X\in\fpr(\cm)$.} Consider the morphism $\epsilon'_X\circ \eta_X\colon \widetilde{X}\to \widetilde{X}'$ in $\cf$. There is a morphism of conflations in $\cf$ by Lemma~\ref{lem:lifting-morphisms}
\[
\xymatrix{
M^{-1}_X\ar[r]^(0.35){\widetilde{a}_X}\ar[d]^{\psi^{-1}_X} & M^0_X\oplus I(M^{-1}_X)\ar[r]^(0.7){\widetilde{b}_X} \ar[d]^{\widetilde{\psi}^0_X} & \widetilde{X}\ar[d]^{\epsilon'_X\circ \eta_X}\\
M'^{-1}_X\ar[r]^(0.35){\widetilde{a'}_X} & M'^0_X\oplus I(M'^{-1}_X)\ar[r]^(0.7){\widetilde{b'}_X}& \widetilde{X}'.
}
\]
Put $\psi^0_X=(\id,0)\circ \widetilde{\psi}^0_X\circ {\id\choose 0}$ and $\psi_X=(\psi^{-1}_X,\psi^0_X)\colon \mathbb{P}'(X)\to\mathbb{P}(X)$. Just as in (8) of the proof of Proposition~\ref{prop:from-fpr-to-2term}, we can show that $\varphi_X\circ\psi_X=\id$ in $\ch^{[-1,0]}(\cm)$ and $\psi_X\circ\varphi_X=\id$ in $\ch^{[-1,0]}(\cm)$. So $\varphi_X$ is an isomorphism.
\end{remark}

\subsection{Extensions (of degree $1$)}
\label{ss:P-detects-the-extriangle-structure}

In this subsection we show that the functor $\mathbb{P}$ detects extriangle structures.

\smallskip

First we have the following special case of Lemma~\ref{lem:taking-presentation-preserves-extension}.

\begin{lemma}
\label{lem:taking-presentation-preserves-extension-2}
The assignment $f\mapsto \Sigma b_Y\circ\Sigma f\circ c_X$ induces a bijection
\[
\Hom_{\ch^b(\cm)}(\mathbb{P}(X),\Sigma \mathbb{P}(Y))\longrightarrow [\Sigma\cm](X,\Sigma Y).
\]
For $g\in [\Sigma\cm](X,\Sigma Y)$, we denote by $\mathbb{P}(g)$ its preimage under this bijection.
\end{lemma}

We will use the following two remarks implicitly.

\begin{remark}
Keep the notation in Remark~\ref{rem:independence-of-the-choice-of-presentation}. We claim that the isomorphism $\varphi\colon\mathbb{P}'\to\mathbb{P}$ is compatible with the bijection in Lemma~\ref{lem:taking-presentation-preserves-extension-2}, \ie for any morphism $g\colon X\to\Sigma Y$ in $\fpr(\cm)$, the following diagram is commutative
\[
\xymatrix{
\mathbb{P}'(X)\ar[r]^{\varphi_X} \ar[d]^{\mathbb{P}'(g)}& \mathbb{P}(X)\ar[d]^{\mathbb{P}(g)}\\
\Sigma\mathbb{P}'(Y)\ar[r]^{\Sigma\varphi_Y} & \Sigma\mathbb{P}(Y).
}
\]
We have
\begin{align*}
\Sigma b_Y\circ \Sigma(\mathbb{P}(g)\circ\varphi^{-1}_X)\circ c'_X&=\Sigma b_Y\circ \Sigma\mathbb{P}(g)\circ (\Sigma \varphi^{-1}_X\circ c'_X)\stackrel{\eqref{cd:phi}}{=}\Sigma b_Y\circ\Sigma\mathbb{P}(g)\circ c_X=g,\\
\Sigma b_Y\circ \Sigma(\varphi^0_X\circ\mathbb{P}'(g))\circ c'_X&=\Sigma (b_Y\circ \varphi^0_X)\circ\Sigma\mathbb{P}'(g)\circ c'_X\stackrel{\eqref{cd:phi}}{=}\Sigma b'_Y\circ\Sigma\mathbb{P}'(g)\circ c'_X=g.
\end{align*}
By Lemma~\ref{lem:taking-presentation-preserves-extension} we have $\mathbb{P}(g)\circ\varphi^{-1}_X=\varphi^0_X\circ\mathbb{P}'(g)$, namely, the above diagram is commutative.
\end{remark}

\begin{remark}
Let $X,Y,Z\in\fpr(\cm)$ and $g\in [\Sigma\cm](X,\Sigma Y)$, $h\in\Hom_\ct(Z,X)$. We claim that $\mathbb{P}(g\circ h)=\mathbb{P}(g)\circ\mathbb{P}(h)$. Let $(h^{-1},h^0)=\mathbb{P}(h)$ and $f=\mathbb{P}(g)$. Then we have the following diagram
\[
\xymatrix{
M^{-1}_Z\ar[r] \ar[d]^{h^{-1}}& M^0_Z\ar[r]\ar[d]^{h^0} & Z\ar[r]^(0.4){c_Z}\ar[d]^h & \Sigma M^{-1}_Z \ar[d]^{\Sigma h^{-1}}\\
M^{-1}_X\ar[r] & M^0_X\ar[r] & X\ar[r]^(0.4){c_X} \ar[d]|(0.3){g} & \Sigma M^{-1}_X\ar@{..>}[dll]|(0.3){\Sigma f}\\
\Sigma M^{-1}_Y\ar[r] & \Sigma M^0_Y\ar[r]_(0.55){-\Sigma b_Y} & \Sigma Y\ar[r]_(0.45){-\Sigma c_Y} & \Sigma^2 M^{-1}_Y.
}
\]
So $g\circ h=\Sigma b_Y\circ \Sigma f \circ c_X\circ h=\Sigma b_Y\circ \Sigma f\circ \Sigma h^{-1}\circ c_Z=\Sigma b_Y\circ \Sigma (f\circ h^{-1})\circ c_Z$. Note that $f\circ h^{-1}=\mathbb{P}(g)\circ \mathbb{P}(h)$. The claim follows.
\end{remark}

\begin{proposition}
\label{prop:P-detects-extriangles}
The functor $\mathbb{P}$ detects extriangles. Precisely, if $\xymatrix@=0.9pc{Y\ar[r]^u & Z\ar[r]^v & X\ar[r]^w &\Sigma Y}$ is an extriangle in $\fpr(\cm)$, then $\xymatrix@=1.2pc{\mathbb{P}(Y)\ar[r]^{\mathbb{P}(u)} & \mathbb{P}(Z)\ar[r]^{\mathbb{P}(v)} & \mathbb{P}(X)\ar[r]^{\mathbb{P}(w)} &\Sigma \mathbb{P}(Y)}$ is an extriangle in $\ch^{[-1,0]}(\cm)$; up to isomorphism every extriangle in $\ch^{[-1,0]}(\cm)$ can be uniquely obtained in this way. 
\end{proposition}
\begin{proof} Assume $\ct=\underline{\cf}$, where $\cf$ is a Frobenius $k$-category, as in the proof of Proposition~\ref{prop:from-fpr-to-2term}.

(a) We will define a $k$-linear functor $\mathbb{P}'\colon \fpr(\cm)\to\ch^{[-1,0]}(\cm)$ which is isomorphic to $\mathbb{P}$ and prove the statement for $\mathbb{P}'$.

Recall that $\widetilde{\cm}$ is the full subcategory of $\cf$ which has the same objects as $\cm$. Let $\fpr_{\cf}(\widetilde{\cm})$ be the full subcategory of $\cf$ consisting of objects $X$ for which there is a conflation $\xymatrix@C=1.3pc{\widetilde{M}^{-1}_X\ar[r]^{\alpha_X} & \widetilde{M}^0_X\ar[r]^{\beta_X} & X}$ with $\widetilde{M}^{-1}_X,~\widetilde{M}^0_X\in\widetilde{\cm}$. For any $X\in\fpr_{\cf}(\widetilde{\cm})$ fix such a conflation
 and define $F(X)\in\ch^{[-1,0]}(\cm)$ as the 2-term complex $(\xymatrix@C=1.3pc{\widetilde{M}^{-1}_X\ar[r]^{\alpha_X}&\widetilde{M}^0_X})$. For any morphism $f\colon X\to Y$ in $\fpr_{\cf}(\widetilde{\cm})$, by Lemma~\ref{lem:lifting-morphisms} there is a morphism of conflations
\[
\xymatrix{
\widetilde{M}^{-1}_X \ar[r]^{\alpha_X} \ar[d]^{f^{-1}} & \widetilde{M}^0_X \ar[r]^{\beta_X} \ar[d]^{f^0} & X\ar[d]^f \\
\widetilde{M}^{-1}_Y \ar[r]^{\alpha_Y} & \widetilde{M}^0_Y \ar[r]^{\beta_Y} & Y.
}
\]
By defining $F(f)$ as the homotopy class of the chain map $(f^{-1},f^0)$ we obtain a $k$-linear functor $F\colon\fpr_{\cf}(\widetilde{\cm})\to\ch^{[-1,0]}(\cm)$, just as in the proof of Proposition~\ref{prop:from-fpr-to-2term}. This functor clearly kills the projective-injective objects in $\cf$, and therefore induces a functor $\underline{F}\colon \underline{\fpr_{\cf}(\widetilde{\cm})}\to\ch^{[-1,0]}(\cm)$. By definition $\underline{\fpr_{\cf}(\widetilde{\cm})}$ is a full subcategory of $\fpr(\cm)$, and then by the diagram~\eqref{cd:eta} in the proof of Proposition~\ref{prop:from-fpr-to-2term}, the inclusion $\underline{\fpr_{\cf}(\widetilde{\cm})}\to\fpr(\cm)$ is an equivalence.  Fix a quasi-inverse $G\colon\fpr(\cm)\to \underline{\fpr_{\cf}(\widetilde{\cm})}$, and put $\mathbb{P}'=\underline{F}\circ G\colon\fpr(\cm)\to\ch^{[-1,0]}(\cm)$. The conflation $\xymatrix@C=1.3pc{\widetilde{M}^{-1}_X\ar[r]^{\alpha_X} & \widetilde{M}^0_X\ar[r]^{\beta_X} & X}$ extends to a triangle $\xymatrix@C=1.3pc{\widetilde{M}^{-1}_X\ar[r]^{\alpha_X} & \widetilde{M}^0_X\ar[r]^(0.55){\beta_X} & X\ar[r]^(0.4){\gamma_X} & \Sigma\widetilde{M}^{-1}_X}$ in $\ct$, where $\gamma_X$ is obtained by a morphism of conflations
\[
\xymatrix@C=3pc{
\widetilde{M}^{-1}_X\ar[r]^{\alpha_X}\ar@{=}[d] & \widetilde{M}^0_X \ar[r]^{\beta_X} \ar[d]& X\ar[d]^{-\gamma_X}\\
\widetilde{M}^{-1}_X\ar[r]^(0.45){i(\widetilde{M}^{-1}_X)} & I(\widetilde{M}^{-1}_X)\ar[r]^{p(\widetilde{M}^{-1}_X)} & \Sigma\widetilde{M}^{-1}_X.
}
\]
If we use these triangles to go through the definition of $\mathbb{P}$ in the proof of Proposition~\ref{prop:from-fpr-to-2term} and  choose $\eta$ and $\epsilon$ canonically, then we obtain $\mathbb{P}'$. 
So by Remark~\ref{rem:independence-of-the-choice-of-presentation}, $\mathbb{P}'$ and $\mathbb{P}$ are isomorphic $k$-linear functors.

Now assume that the extriangle $\xymatrix@=0.9pc{Y\ar[r]^u & Z\ar[r]^v & X\ar[r]^w &\Sigma Y}$ comes from the following morphism of conflations
\[
\xymatrix{
Y\ar[r]^u \ar@{=}[d] & Z\ar[r]^v\ar[d]^z & X\ar[d]^{-w}\\
Y\ar[r]^(0.45){i(Y)} & I(Y) \ar[r]^{p(Y)} & \Sigma Y.
}
\]
In $\ct=\underline{\cf}$ the morphism $w$ factors through an object in $\Sigma\cm$ , so $w\circ \beta_X=0$ in $\ct$ because $\cm$ is rigid, namely, $w\circ\beta_X$ factors through a projective-injective object in $\cf$. Therefore it factors through $p(Y)$ to produce a morphism $f'\colon\widetilde{M}^0_X\to I(Y)$ such that $-w\circ\beta_X=p(Y)\circ f'$. Because the lower right square of \eqref{cd:push-out} is a pull-back diagram, there exists $f\colon\widetilde{M}^0_X\to Z$ such that $f'=z\circ f$ and $\beta_X=v\circ f$. It follows that there exists $g'\colon \widetilde{M}^{-1}_X\to Y$ such that the upper vertical morphisms of the following diagram \eqref{cd:push-out} form a morphism of conflations
\begin{align}
\label{cd:push-out}
\xymatrix@C=4pc{\widetilde{M}^{-1}_X\ar[r]^{\alpha_X}\ar@{-->}[d]^{g'} & \widetilde{M}^0_X \ar[r]^{\beta_X} \ar@{-->}@/^1.5pc/[dd]|(0.3){f'} \ar@{..>}[d]_f& X\ar@{=}[d]\\
Y\ar[r]^u \ar@{=}[d] & Z\ar[r]^v\ar[d]_z & X\ar[d]^{-w}\\
Y\ar[r]^{i(Y)} & I(Y) \ar[r]^{p(Y)} & \Sigma Y.
}
\end{align}
Then using the isomorphism $\Ext^1_\cf(X,Y)\cong \Hom_\ct(X,\Sigma Y)$ we know that $w=\Sigma g'\circ \gamma_X$ in $\ct$. Now because $\Ext^1_\cf(\widetilde{M}^{-1}_X,\widetilde{M}^{-1}_Y)=0$, there exists $g\colon \widetilde{M}^{-1}_X\to\widetilde{M}^0_Y$ such that $g'=\beta_Y\circ g$
\[
\xymatrix{
&&\widetilde{M}^{-1}_X\ar[d]^{g'}\ar@{-->}[ld]_g\\
\widetilde{M}^{-1}_Y \ar[r]^{\alpha_Y} & \widetilde{M}^0_Y\ar[r]^{\beta_Y} & Y.
}
\]
So $w=\Sigma\beta_Y\circ \Sigma g\circ\gamma_X$ in $\ct$. Thus by Lemma~\ref{lem:taking-presentation-preserves-extension-2} applied to $\mathbb{P}'$, we have $g=\mathbb{P}'(w)$.

Moreover, we obtain two morphisms of conflations
\[
\xymatrix@C=4pc{
\widetilde{M}^{-1}_Y\ar[r]^{0\choose \id} \ar[d]^{\alpha_Y} & \widetilde{M}^{-1}_X \oplus\widetilde{M}^{-1}_Y \ar[r]^{(\id,0)} \ar[d]^{\begin{psmallmatrix}\alpha_X & 0\\ -g & \alpha_Y\end{psmallmatrix}} & \widetilde{M}^{-1}_X\ar[d]^{\alpha_X}\\
\widetilde{M}^0_Y\ar[r]^{0\choose \id} \ar[d]^{\beta_Y} & \widetilde{M}^0_X\oplus\widetilde{M}^0_Y \ar[r]^{(\id,0)}\ar[d]^{(f,u\circ\beta_Y)} & \widetilde{M}^0_X\ar[d]^{\beta_X}\\
Y\ar[r]^u & Z\ar[r]^v & X,
}
\]
where the left and right columns are also conflations. So by \cite[Corollary 3.2]{Buehler10}, the top morphism in the middle column is an inflation and the bottom one is a deflation. We claim that $(f,u\circ\beta_Y)=\cok\left(\begin{array}{cc}\small\alpha_X & 0\\ -g & \alpha_Y\end{array}\right)$. Indeed, it is clear that 
\[
(f,u\circ\beta_Y)\left(\begin{array}{cc}\small\alpha_X & 0\\ -g & \alpha_Y\end{array}\right)=0.
\]
Moreover, assume that $h\colon Z\to U$ satisfies $h\circ f=0$ and $h\circ u\circ \beta_Y=0$. The latter implies that $h\circ u=0$ since $\beta_Y$ is an epimorphism. So $h=0$ because the upper left square of the diagram~\eqref{cd:push-out} is a push-out diagram. So the middle column of the above diagram is a conflation. Therefore there are morphisms of conflations by Lemma~\ref{lem:lifting-morphisms}
\[
\xymatrix@C=4pc{
\widetilde{M}^{-1}_Z\ar[r]^{\alpha_Z}\ar[d]^{f^{-1}} & \widetilde{M}^0_Z\ar[r]^{\beta_Z} \ar[d]^{f^0}& Z\ar@{=}[d]\\
\widetilde{M}^{-1}_X\oplus\widetilde{M}^{-1}_Y\ar[r]^{\begin{psmallmatrix}\alpha_X & 0\\ -g & \beta_X\end{psmallmatrix}}\ar[d]^{g^{-1}} & \widetilde{M}^0_X\oplus\widetilde{M}^0_Y \ar[r]^(0.55){(f,u\circ\beta_Y)} \ar[d]^{g^0}& Z\ar@{=}[d]\\
\widetilde{M}^{-1}_Z\ar[r]^{\alpha_Z} & \widetilde{M}^0_Z\ar[r]^{\beta_Z} & Z
}
\]
such that $(g^{-1}f^{-1},g^0f^0)$ and $(f^{-1}g^{-1},f^0g^0)$ are both homotopy equivalent to $\id$, and there are morphisms of conflations
\[
\xymatrix@C=4pc{
\widetilde{M}^{-1}_Y\ar[r]^{\alpha_Y}\ar[d]^{g^{-1}{0\choose\id}} & \widetilde{M}^0_Y\ar[r]^{\beta_Y}\ar[d]^{g^0{0\choose \id}} & Y\ar[d]^u\\
\widetilde{M}^{-1}_Z\ar[r]^{\alpha_Z}\ar[d]^{(\id,0)f^{-1}} & \widetilde{M}^0_Z\ar[r]^{\beta_Z}\ar[d]^{(\id,0)f^0} & Z\ar[d]^v\\
\widetilde{M}^{-1}_X\ar[r]^{\alpha_X} & \widetilde{M}^0_X\ar[r]^{\beta_X} & X.
}
\] 
This shows that there is an isomorphism of $\Sigma$-sequences in $\ch^b(\cm)$:
\[
\begin{CD}
(\begin{CD} \widetilde{M}^{-1}_Y @>{\alpha_Y}>> \widetilde{M}^0_Y \end{CD}) @= (\begin{CD} \widetilde{M}^{-1}_Y@>{\alpha_Y}>>  \widetilde{M}^0_Y \end{CD})\\ 
@VV\left({0\choose \id},{0\choose \id}\right)V @VV{\mathbb{P}'(u)}V\\
\begin{CD} (\widetilde{M}^{-1}_X\oplus\widetilde{M}^{-1}_Y @>{\begin{psmallmatrix}\alpha_X & 0\\ -g & \beta_X\end{psmallmatrix}}>> \widetilde{M}^0_X\oplus\widetilde{M}^0_Y) \end{CD} @>(g_1,g_0)>> (\begin{CD} \widetilde{M}^{-1}_Z @>{\alpha_Z}>>  \widetilde{M}^0_Z \end{CD})\\
@VV((\id,0),(\id,0))V @VV{\mathbb{P}'(v)}V\\
(\begin{CD} \widetilde{M}^{-1}_X @>{\alpha_X}>>  \widetilde{M}^0_X \end{CD}) @= (\begin{CD} \widetilde{M}^{-1}_X @>{\alpha_X}>>  \widetilde{M}^0_X \end{CD})\\
@VVgV @VV{\mathbb{P}'(w)}V\\
\Sigma(\begin{CD} \widetilde{M}^{-1}_Y @>{\alpha_Y}>> \widetilde{M}^0_Y \end{CD}) @= \Sigma(\begin{CD} \widetilde{M}^{-1}_Y @>{\alpha_Y}>> \widetilde{M}^0_Y \end{CD})
\end{CD}
\]
So the sequence on the right is an extriangle, as desired.

(b) Let $\xymatrix@=0.9pc{Y'\ar[r]^g & Z'\ar[r]^h & X'\ar[r]^f &\Sigma Y'}$ be an extriangle in $\ch^{[-1,0]}(\cm)$. Since $\mathbb{P}$ is dense, there are $X,Y\in\fpr(\cm)$ and isomorphisms $x\colon \mathbb{P}(X)\to X'$ and $y\colon\mathbb{P}(Y)\to Y'$. By Lemma~\ref{lem:taking-presentation-preserves-extension-2} there is a unique $w\in[\Sigma\cm](X,\Sigma Y)$ such that the following diagram is commutative
\[
\xymatrix{
\mathbb{P}(X)\ar[r]^{\mathbb{P}(w)} \ar[d]^x& \Sigma \mathbb{P}(Y)\ar[d]^{\Sigma y}\\
X'\ar[r]^f & \Sigma Y'.
}
\]
Complete $w$ to a triangle in $\ct$
\[
\xymatrix{
Y\ar[r]^u & Z\ar[r]^v & X\ar[r]^w & \Sigma Y.
}
\]
This is an extriangle in $\fpr(\cm)$ by Lemma~\ref{lem:the-extriangle-structure}. 
Applying $\mathbb{P}$ we obtain an isomorphism of extriangles
\[
\xymatrix{
\mathbb{P}(Y)\ar[r]^{\mathbb{P}(u)}\ar[d]^y & \mathbb{P}(Z)\ar[r]^{\mathbb{P}(v)}\ar[d]^\cong & \mathbb{P}(X)\ar[r]^{\mathbb{P}(w)} \ar[d]^x & \Sigma\mathbb{P}(Y)\ar[d]^{\Sigma y}\\
Y'\ar[r]^g & Z' \ar[r]^h & X'\ar[r]^f & \Sigma Y'.
}
\]
Here that the first row is an extriangle follows from (a).
\end{proof}

\subsection{Morphisms}
\label{ss:the-kernel-of-P}
In this subsection we show that $\mathbb{P}$ is full and describe its kernel.

\smallskip

Let $\ci$ be the ideal of $\fpr(\cm)$ consisting of  morphisms which factors through a morphism $\Sigma M\to M'$ with $M,M'\in \cm$. Then $\ci^2=0$ since $\cm$ is rigid. 

\begin{theorem}
\label{thm:the-kernel}
The functor $\mathbb{P}$ is full, dense and induces an equivalence $\frac{\fpr(\cm)}{\ci}\to \ch^{[-1,0]}(\cm)$.
In particular, $\mathbb{P}$ detects isomorphisms, detects indecomposability and induces a bijection between the isomorphism classes of objects (respectively, indecomposable objects) of $\fpr(\cm)$ and those of $\ch^{[-1,0]}(\cm)$.
\end{theorem}
\begin{proof} Assume $\ct=\underline{\cf}$, where $\cf$ is a Frobenius category, as in the proof of Proposition~\ref{prop:from-fpr-to-2term}.

(1) \emph{$\mathbb{P}$ is dense.} By Proposition~\ref{prop:from-fpr-to-2term}.

(2) \emph{$\mathbb{P}$ is full.}  Let $X,Y\in\fpr(\cm)$ and let $g=(g^{-1},g^0)\colon\mathbb{P}(X)\to\mathbb{P}(Y)$ be a chain map
\[
\xymatrix{
M^{-1}_X\ar[r]^{a_X}\ar[d]^{g^{-1}} & M^0_X\ar[d]^{g^0}\\
M^{-1}_Y\ar[r]^{a_Y} & M^0_Y.
}
\]
Then there is a commutative diagram in the Frobenius category $\cf$
\[
\xymatrix{
M^{-1}_X\ar[r]^(0.35){\widetilde{a}_X}\ar[d]^{g^{-1}} & M^0_X\oplus I(M^{-1}_X)\ar[d]^{\widetilde{g}^0}\\
M^{-1}_Y\ar[r]^(0.35){\widetilde{a}_Y} & M^0_Y\oplus I(M^{-1}_Y).
}
\]
We complete it to a morphism of conflations in $\cf$
\[
\xymatrix{
M^{-1}_X\ar[r]^(0.35){\widetilde{a}_X}\ar[d]^{g^{-1}} & M^0_X\oplus I(M^{-1}_X)\ar[d]^{\widetilde{g}^0}\ar[r]^(0.7){\widetilde{b}_X} & \widetilde{X}\ar[d]^{\widetilde{f}}\\
M^{-1}_Y\ar[r]^(0.35){\widetilde{a}_Y} & M^0_Y\oplus I(M^{-1}_Y)\ar[r]^(0.7){\widetilde{b}_Y} &\widetilde{Y}.
}
\]
Let $f=\epsilon_X\circ\widetilde{f}\circ\eta_Y$. Then $(g^{-1},g^0)$ is a liftable presentation of $f$ and thus $g=\mathbb{P}(f)$.

(3) \emph{Let $f\colon X\to Y$ be a morphism in $\fpr(\cm)$. Then $\mathbb{P}(f)=0$ if and only if $f\in\ci$.}
Assume $f\in\ci$. Then $\mathbb{P}(f)=0$ because $\Hom_{\ch^{[-1,0]}(\cm)}(\Sigma \mathbb{P}(M),\mathbb{P}(M'))=0$ for any $M,M'\in\cm$. Next assume $\mathbb{P}(f)=0$. 

\noindent Case 1: $X=M\in\cm$. Consider a liftable presentation of $f$
\[
\xymatrix{
0\ar[r] \ar[d] & M \ar[r]^{\id} \ar[d]^{f^0} & M\ar[r]\ar[d]^f & 0\ar[d]\\
M^{-1}_Y\ar[r]^{a_Y} & M^0_Y\ar[r]^{b_Y} & Y\ar[r]^{c_Y} & \Sigma M^{-1}_Y.
}
\] 
Since $\mathbb{P}(f)=0$, there is a morphism $h\colon M\to M_1^Y$ such that $f^0=a_Y\circ h$. So $f=b_Y\circ f^0=0$ belongs to $\ci$.

\noindent Case 2: $X=\Sigma M\in\Sigma\cm$. Consider a liftable presentation of $f$
\[
\xymatrix{
M\ar[r] \ar[d]^{f^{-1}} & 0 \ar[r] \ar[d]^{0} & \Sigma M\ar[r]^{\id}\ar[d]^f & \Sigma M\ar[d]^{\Sigma f^{-1}}\\
M^{-1}_Y\ar[r]^{a_Y} & M^0_Y\ar[r]^{b_Y} & Y\ar[r]^{c_Y} & \Sigma M^{-1}_Y.
}
\]
Since $\mathbb{P}(f)=0$, it follows that $f^{-1}=0$. Therefore $c_Y\circ f=0$, and hence $f$ factors through $b_Y$ and belongs to $\ci$.

\noindent Case 3: the general case. We have $\mathbb{P}(f\circ b_X)=\mathbb{P}(f)\circ \mathbb{P}(b_X)=0$, so by Case 1, $f\circ b_X=0$. Therefore $f$ factors through $c_X$ to produce a morphism $f_1\colon \Sigma M^{-1}_X\to Y$ such that $f=f_1\circ c_X$. 
\[
\xymatrix@C=4pc{
M^{-1}_X\ar[r]^{a_X} & M^0_X \ar[r]^{b_X} & X\ar[r]^{c_X} \ar[d]|(0.4)f& \Sigma M^{-1}_X\ar@{.>}[dl]|(0.4){f_1}\ar[r]^{-\Sigma a_X} & \Sigma M^0_X\ar@{.>}[dll]|(0.4){f_3}\\
M^{-1}_Y\ar[r]^{a_Y} & M^0_Y\ar[r]^{b_Y} & Y\ar[r]^{c_Y} & \Sigma M^{-1}_Y
}
\]
Now $\mathbb{P}(f_1)\circ\mathbb{P}(c_X)=\mathbb{P}(f_1\circ c_X)=\mathbb{P}(f)=0$, so there exists $f_2\colon\Sigma\mathbb{P}(M_0^X)\to\mathbb{P}(Y)$ such that $\mathbb{P}(f_1)=-f_2\circ\Sigma\mathbb{P}(a_X)$. 
\[
\xymatrix@C=4pc{
\mathbb{P}(M^{-1}_X)\ar[r]^{\mathbb{P}(a_X)} & \mathbb{P}(M^0_X) \ar[r]^{\mathbb{P}(b_X)} & \mathbb{P}(X)\ar[r]^{\mathbb{P}(c_X)} \ar[d]|(0.4){\mathbb{P}(f)=0}& \Sigma \mathbb{P}(M^{-1}_X)\ar@{.>}[dl]|(0.4){\mathbb{P}(f_1)}\ar[r]^{-\Sigma\mathbb{P}(a_X)} & \Sigma\mathbb{P}(M^0_X)\ar@{.>}[dll]|(0.4){f_2}\\
\mathbb{P}(M^{-1}_Y)\ar[r]_{\mathbb{P}(a_Y)} & \mathbb{P}(M^0_Y)\ar[r]_{\mathbb{P}(b_Y)} & \mathbb{P}(Y)\ar[r]_{\mathbb{P}(c_Y)} & \Sigma \mathbb{P}(M^{-1}_Y)
}
\]
By (2), there exists $f_3\colon\Sigma M^0_X\to Y$ such that $\mathbb{P}(f_3)=f_2$. Let $f_1'=f_1+f_3\circ\Sigma a_X$. Then $\mathbb{P}(f_1')=0$. By Case 2, $f_1'$ factors through $b_Y$ to produce a morphism $a\colon \Sigma M^{-1}_X\to M^0_Y$ such that $f_1'=b_Y\circ a$. Therefore
$f=f_1\circ c_X=(f_1'-f_3\circ\Sigma a_X)\circ c_X=f_1'\circ c_X=b_Y\circ a\circ c_X$, as desired.

(4) \emph{$\mathbb{P}$ induces an equivalence $\frac{\fpr(\cm)}{\ci}\to \ch^{[-1,0]}(\cm)$.} This immediately follows from (1) (2) and (3).

(5) \emph{Let $f\colon X\to Y$ be a morphism in $\fpr(\cm)$. Then $f$ is an isomorphism if and only if $\mathbb{P}(f)$ is an isomorphism.} Assume that $\mathbb{P}(f)$ is an isomorphism. Then there exists $g'\colon \mathbb{P}(Y)\to\mathbb{P}(X)$ such that $\mathbb{P}(f)\circ g'=\id_{\mathbb{P}(Y)}$ and $g'\circ\mathbb{P}(f)=\id_{\mathbb{P}(X)}$. Let $f'\colon Y\to X$ be a morphism such that $g'=\mathbb{P}(f')$. Then $f\circ f'\in \id_Y+\ci$ and $f'\circ f\in\id_X+\ci$ are isomorphisms because $\ci^2=0$, and it follows that $f$ is an isomorphism.

(6) \emph{Let $X\in\fpr(\cm)$. Then $X$ is indecomposable if and only if $\mathbb{P}(X)$ is indecomposable.} Assume that $e$ is a non-trivial idempotent of $\End_\ct(X)$. Since $\ci^2=0$, it follows that $e$ does not belong to $\ci$ and thus $\mathbb{P}(e)\neq 0$. For the same reason $\id_{\mathbb{P}(X)}-\mathbb{P}(e)=\mathbb{P}(1-e)\neq 0$. Therefore $\mathbb{P}(e)$ is a non-trivial idempotent of $\End(\mathbb{P}(X))$. Conversely, assume that $f$ is a non-trivial idempotent of $\End(\mathbb{P}(X))$. Then by \cite[Lemma 3.2.1]{DrozdKirichenko}, there exists a non-trivial idempotent $e$ of $\End_\ct(X)$ such that $\mathbb{P}(e)=f$. In particular, $\End_\ct(X)$ has a non-trivial idempotent.

(7) \emph{$\mathbb{P}$ induces a bijection from the set of isomorphism classes of objects of $\fpr(\cm)$ to that of $\ch^{[-1,0]}(\cm)$.} This follows from (4) and (5).

(8) \emph{$\mathbb{P}$ induces a bijection from the set of isomorphism classes of indecomposable objects of $\fpr(\cm)$ to that of $\ch^{[-1,0]}(\cm)$.} This follows from (7) and (6).
\end{proof}

The following corollary is immediate.
\begin{corollary}
\label{cor:when-P-is-an-equivalence}
The functor $\mathbb{P}$ is an equivalence if and only if $\Hom_\ct(M,\Sigma^{-1}M')=0$ for any $M,M'\in\cm$.
\end{corollary}

\subsection{Extensions of degree $-1$}

In this subsection we study the behaviour of $\mathbb{P}$ on extensions of degree $-1$. 

\smallskip
For $X,Y\in\fpr(\cm)$, the set $\Hom(\mathbb{P}(X),\Sigma^{-1}\mathbb{P}(Y))$ consists of morphisms $g\colon M^0_X\to M^{-1}_Y$ such that $g\circ a_X=0=a_Y\circ g$. The functor $\mathbb{P}$ induces the following map
\[
\xymatrix@R=0.5pc{
\rho\colon\Hom(X,\Sigma^{-1}Y)\ar[r] & \Hom(\mathbb{P}(X),\Sigma^{-1}\mathbb{P}(Y)).\\
f\ar@{|->}[r] & -\Sigma^{-1}c_Y\circ f\circ b_X
}
\]
\[
\xymatrix@C=4pc{
M^{-1}_X\ar[r]^{a_X} & M^0_X\ar[r]^{b_X} & X\ar[r]^{c_X}\ar[d]^f & \Sigma M^{-1}_X\\
\Sigma^{-1}M^{-1}_Y\ar[r]^{-\Sigma^{-1}a_Y} & \Sigma^{-1}M^0_Y\ar[r]^{-\Sigma^{-1}b_Y} & \Sigma^{-1}Y\ar[r]^{-\Sigma^{-1}c_Y} & M^{-1}_Y.
}
\]

\begin{lemma}
\label{lem:P-on-negative-extensions}
Assume that $\Hom_\ct(M,\Sigma^{-1}M')=0$ for any $M,M'\in\cm$. Then $\rho$ is sujective with kernel those morphisms factoring a morphism $\Sigma M\to\Sigma^{-1}M'$ with $M,M'\in\cm$.
\end{lemma}
\begin{proof}
We first describe the kernel of $\rho$. Let $f$ be in the kernel of $\rho$, \ie $-\Sigma^{-1}c_Y\circ f\circ b_X=0$. Then $f\circ b_X$ factors through $-\Sigma^{-1}b_Y$. But $\Hom_\ct(M^0_X,\Sigma^{-1}M^0_Y)=0$ by assumption, therefore $f\circ b_X=0$. So $f$ factors through $c_X$ to produce a  morphism $f'\colon \Sigma M^{-1}_X\to \Sigma^{-1}Y$ such that $f=f'\circ c_X$.
Now $-\Sigma^{-1}c_Y\circ f'=0$ because $\Hom_\ct(\Sigma M^{-1}_X,M^{-1}_Y)=0$ by assumption, so $f'$ factors through $-\Sigma^{-1}b_Y$ to produce a morphism $f''\colon \Sigma M^{-1}_X\to\Sigma^{-1}M^0_Y$ such that $f'=-\Sigma^{-1}b_Y\circ f''$. Therefore $f=-\Sigma^{-1}b_Y\circ f''\circ c_X$, as desired.
\[
\xymatrix@C=4pc{
M^{-1}_X\ar[r]^{a_X} & M^0_X\ar[r]^{b_X} & X\ar[r]^{c_X}\ar[d]|(0.3)f & \Sigma M^{-1}_X\ar@{.>}[dl]|{f'}\ar@{.>}[dll]|(0.35){f''}\\
\Sigma^{-1}M^{-1}_Y\ar[r]_{-\Sigma^{-1}a_Y} & \Sigma^{-1}M^0_Y\ar[r]_{-\Sigma^{-1}b_Y} & \Sigma^{-1}Y\ar[r]_{-\Sigma^{-1}c_Y} & M^{-1}_Y.
}
\]
Conversely, let $f\colon X\to\Sigma^{-1}Y$ be a morphism factoring through a morphism $\Sigma M\to \Sigma^{-1} M'$ with $M,M'\in\cm$. Then $f\circ b_X=0$ and thus $f$ belongs to the kernel of $\rho$.

Next we show the surjectivity of $\rho$. Let $g\colon M^0_X\to M^{-1}_Y$ be a morphism such that $g\circ a_X=0=a_Y\circ g$. Since $a_Y\circ g=0$, it follows that $g$ factors through $-\Sigma^{-1}c_Y$ to produce a morphism $g'\colon M^0_X\to\Sigma^{-1}Y$ such that $g=-\Sigma^{-1}c_Y\circ g'$. Now $-\Sigma^{-1}c_Y\circ g'\circ a_X=g\circ a_X=0$, so $g'\circ a_X$ factors through $-\Sigma^{-1}b_Y$, but $\Hom_\ct(M^{-1}_X,\Sigma^{-1}M^0_Y)=0$ by assumption, soe $g'\circ a_X=0$. So $g'$ factors through $b_X$ to produce a morphism $g''\colon X\to\Sigma^{-1}Y$ such that $g'=g''\circ b_X$. It follows that $g=\rho(g'')$.
\[
\xymatrix@C=4pc{
M^{-1}_X\ar[r]^{a_X} & M^0_X\ar[r]^{b_X}\ar[drr]|(0.3)g\ar@{.>}[dr]|(0.3){g'} & X\ar[r]^{c_X}\ar@{.>}[d]|(0.3){g''} & \Sigma M^{-1}_X\\
\Sigma^{-1}M^{-1}_Y\ar[r]_{-\Sigma^{-1}a_Y} & \Sigma^{-1}M^0_Y\ar[r]_{-\Sigma^{-1}b_Y} & \Sigma^{-1}Y\ar[r]_{-\Sigma^{-1}c_Y} & M^{-1}_Y.
}
\]
\end{proof}

\subsection{Compatibility with cluster-tilting}

Consider the map
\[
\mathbf{p}\colon\cn\mapsto\Im(\mathbb{P}|_\cn),
\] 
from the set of full subcategories of $\fpr(\cm)$ closed under isomorphisms to the set of full subcategories of $\ch^{[-1,0]}(\cm)$ closed under isomorphisms. By Theorem~\ref{thm:the-kernel}, $\mathbf{p}$ is bijective.

\begin{theorem}
\label{thm:P-induces-bijection-between-cto} 
The map $\mathbf{p}$ restricts to a bijection from the set of relative rigid subcategories of $\fpr(\cm)$ to the set of 2-term presilting subcategories of $\ch^b(\cm)$, which further restricts to a bijection from the set of weakly relative cluster-tilting subcategories of $\fpr(\cm)$ to the set of 2-term silting subcategories of $\ch^b(\cm)$. 
\end{theorem}
\begin{proof}
Let $\cn$ be a subcategory of $\fpr(\cm)$.  By Lemma~\ref{lem:taking-presentation-preserves-extension-2}, it is relative rigid if and only if $\mathbf{p}(\cn)$ is a rigid subcategory of $\ch^{[-1,0]}(\cm)$, if and only if $\mathbf{p}(\cn)$ is a 2-term presilting subcategory of $\ch^b(\cm)$ by Lemma~\ref{lem:rigid-vs-presilting-in-homotopy-category}. Applying Proposition~\ref{prop:P-detects-extriangles}, we see that $\cn$ is generating if and only if $\mathbf{p}(\cn)$ is generating. Therefore $\cn$ is weakly relative cluster-tilting if and only if $\Im(\mathbb{P}|_\cn)$ is weakly relative cluster-tilting, if and only if $\Im(\mathbb{P}|_\cn)$ is 2-term silting in $\ch^b(\cm)$, by Lemma~\ref{lem:rigid-vs-presilting-in-homotopy-category}.
\end{proof}

\begin{corollary}
A subcategory of $\fpr(\cm)$ is weakly relative cluster-tilting if and only if it is relative rigid and generating.
\end{corollary}

\subsection{The one-object case}

Assume that $\cm=\add(M)$ for a rigid object $M\in\ct$. Let $A=\End_{\ct}(M)$. Then there is the Auslander's projectivisation 
\[
\Hom(M,?)\colon\add(M)\longrightarrow \proj A,
\]
which is an equivalence, and induces a triangle equivalence $\ch^b(\add(M))\to \ch^b(\proj A)$. Composing this triangle equivalence with $\mathbb{P}$, we obtain a functor, still denoted by $\mathbb{P}$:
\[
\mathbb{P}\colon \fpr(M)\longrightarrow\ch^{[-1,0]}(\add(M))\longrightarrow\ch^{[-1,0]}(\proj A).
\]

\begin{theorem}
\label{thm:the-one-object-case}
\begin{itemize}
\item[(a)] The functor $\mathbb{P}$ is full, dense and induces an equivalence $\frac{\fpr(M)}{\ci}\to \ch^{[-1,0]}(\proj A)$, where $\ci$ consists of morphisms which factors through a morphism $\Sigma M_1\to M_2$ with $M_1,M_2\in\add(M)$. In particular, $\mathbb{P}$ is an equivalence if and only if $\Hom_\ct(M,\Sigma^{-1}M)=0$.
\item[(b)] $\mathbb{P}$ detects isomorphisms, detects indecomposability and induces a bijection between the isomorphism classes of objects of $\fpr(M)$ and those of $\ch^{[-1,0]}(\proj A)$.
\item[(c)] $\mathbb{P}$ induces a bijection between the set of isomorphism classes of  relative cluster-tilting objects of $\fpr(M)$ and the set of isomorphism classes of 2-term silting objects of $\ch^b(\proj A)$. 
\item[(d)] Assume further that $k$ is a field and $\fpr(M)$ is Hom-finite over $k$. Then the bijection in (c) restricts a bijection between the set of isomorphism classes of basic relative cluster-tilting objects of $\fpr(M)$ and the set of isomorphism classes of basic 2-term silting objects of $\ch^b(\proj A)$, which commutes with mutations.
\end{itemize}
\end{theorem}
\begin{proof}
It remains to prove that the bijection in (d) commutes with mutations. Suppose that $\mu_i^L(N)=\bigoplus_{j\neq i}N_j\oplus N_i^*$ is a left mutation of a basic relative cluster-tilting object $N=\bigoplus_{i=1}^n N_i$. Then there is an extriangle in $\fpr(\cm)$
\begin{align*}
\xymatrix{
N_i\ar[r] & E \ar[r] & N_i^*\ar[r]&\Sigma N_i ,
}
\end{align*}
so by Proposition~\ref{prop:P-detects-extriangles}, there is an extriangle in $\ch^{[-1,0]}(\cm)$
\begin{align*}
\xymatrix{
\mathbb{P}(N_i)\ar[r] & \mathbb{P}(E) \ar[r] & \mathbb{P}(N_i^*)\ar[r]&\Sigma \mathbb{P}(N_i) .
}
\end{align*}
Moreover, $\mathbb{P}(\mu_i^L(N))$ is a basic relative cluster-tilting object and both $\mathbb{P}(N_i)$ and $\mathbb{P}(N_i^*)$ are indecomposable. This implies that $\mathbb{P}(\mu_i^L(N))\cong\mu_i^L(\mathbb{P}(N))$.
\end{proof}

\subsubsection{Special case: $M$ is a silting object}
Assume that $M$ is a silting object of $\ct$. Then a silting object of $\ct$ belonging to $\fpr(M)$ is called a \emph{2-term silting object}.
As a special case of Theorem~\ref{thm:the-one-object-case}, we recover \cite[Propositions A.3, A.5 and A.6]{BruestleYang13} (and Theorem A.7 in its arXiv version arXiv:1302.6045v6)\footnote{In this footnote we adopt the notation in \cite[Appendix A]{BruestleYang13}. Fix a cofibrant resolution for each object in $\add_{\per(A)}(A)$ and inductively construct a cofibrant resolution for any object in $\per(A)$ by forming cones and shifts. Take the category of these cofibrant resolutions as the Frobenius model of $\per(A)$. Then the functor $p_*|_{\cf_A}\colon\cf_A\to\cf_{\bar{A}}$ there is exactly the functor $\mathbb{P}$ here. We point out that the proof of \cite[Proposition A.3]{BruestleYang13} is not complete, and it is completed in the arXiv version arXiv:1302.6045v6.} and \cite[Proposition 4.22]{HuaKeller24}.

\begin{theorem}
\label{thm:the-one-object-case-silting}
\begin{itemize}
\item[(a)] The functor $\mathbb{P}$ is full, dense and induces an equivalence $\frac{\fpr(M)}{\ci}\to \ch^{[-1,0]}(\proj A)$, where $\ci$ consists of morphisms which factors through a morphism $\Sigma M_1\to M_2$ with $M_1,M_2\in\add(M)$. In particular, $\mathbb{P}$ is an equivalence if and only if $\Hom_\ct(M,\Sigma^{-1}M)=0$.
\item[(b)] $\mathbb{P}$ detects isomorphisms, detects indecomposability and induces a bijection between the isomorphism classes of objects of $\fpr(M)$ and those of $\ch^{[-1,0]}(\proj A)$.
\item[(c)] $\mathbb{P}$ induces a bijection between the sets of isomorphism classes of  2-term silting objects of $\ct$ and of $\ch^b(\proj A)$. 
\item[(d)] Assume further that $k$ is a field and $\fpr(M)$ is Hom-finite over $k$. Then the bijection in (c) restricts a bijection between the sets of isomorphism classes of basic 2-term silting objects of $\ct$ and  of $\ch^b(\proj A)$, which commutes with mutations.
\end{itemize}
\end{theorem}

Lemma~\ref{lem:taking-presentation-preserves-extension-2} and Theorem~\ref{thm:the-kernel} show that the functor $\mathbb{P}$ behave rather well on morphisms and extensions of degree 1. But we do not have a good understanding on its behaviour on extensions of degree $-1$. The following question was proposed by Yu Zhou.

\begin{question}
Let $T\in\fpr(M)$ be a tilting object of $\ct$. Is $\mathbb{P}(T)$ a tilting object of $\ch^b(\proj A)$? 
\end{question}

In the context of finite-dimensional algebras $\mathbb{P}(T)$ is exactly the silting complex $Q$ in \cite[Theorem 1.1]{BuanZhou16}. 
This question is studied in \cite[Section 5]{XieYangZhang23}, in particular, it has a positive answer if $\End_\ct(T)$ is hereditary. 

\subsubsection{Special case: $M$ is a cluster-tilting object}
An object $T$ of $\ct$ is called a \emph{cluster-tilting object} if 
\[
\add(T)=\{X\in\ct|\Hom_\ct(T,\Sigma X)=0\}=\{X\in\ct|\Hom_\ct(X,\Sigma T)=0\}.
\]
In the rest of this subsection we further assume the following conditions:
\begin{itemize}
\item $k$ is a field and $\ct$ is Hom-finite over $k$,
\item $\ct$ is $2$-Calabi--Yau, \ie  there is a bifunctorial isomorphism
\[
D\Hom_{\ct}(X,Y)\stackrel{\cong}{\longrightarrow} \Hom_{\ct}(Y,\Sigma^2 X)
\]
for any $X,Y\in\ct$ (\cite{Kontsevich98,Keller10}), where $D=\Hom_k(?,k)$ (note that in \cite{Keller10} such categories are called weakly $2$-Calabi--Yau triangulated categories);
\item $M$ is a basic cluster-tilting object of $\ct$,
\end{itemize}
In this case $\ct=\fpr(M)$, by \cite[Proposition 2.1]{KellerReiten07} (see also~\cite[Lemma 3.2]{KoenigZhu08}). Thus we have a functor 
\[
\mathbb{P}\colon \ct\longrightarrow \ch^{[-1,0]}(\proj A).
\]

\begin{theorem}
\label{thm:the-one-object-case-cto}
\begin{itemize}
\item[(a)] The functor $\mathbb{P}$ is full, dense and induces an equivalence $\frac{\ct}{\ci}\to \ch^{[-1,0]}(\proj A)$, where $\ci$ is consists of morphisms which factors through a morphism $\Sigma M_1\to M_2$ with $M_1,M_2\in\add(M)$. The functor $\mathbb{P}$ is an equivalence if and only if $A$ is self-injective. 
\item[(b)] $\mathbb{P}$ detects isomorphisms, detects indecomposability and induces a bijection between the isomorphism classes of objects of $\ct$ and those of $\ch^{[-1,0]}(\proj A)$.
\item[(c)] $\mathbb{P}$ induces a bijection between the set of isomorphism classes of basic cluster-tilting objects of $\ct$ and the set of isomorphism classes of basic 2-term silting objects of $\ch^b(\proj A)$, which commutes with mutations.
\item[(d)] Assume that $A$ is self-injective. Then the bijection in (c) restricts to a bijection from (i) to (ii) of the following sets:
\begin{itemize}
\item[(i)] the set of isomorphism classes of basic cluster-tilting objects of $\ct$ with self-injective endomorphism algebras,
\item[(ii)] the set of isomorphism classes of basic 2-term tilting objects in $\ch^b(\proj A)$.
\end{itemize}
As a consequence, a basic 2-term silting object $N$ of $\ch^b(\proj A)$ is tilting if and only if its endomorphism algebra is self-injective.
\end{itemize}
\end{theorem}
As a consequence of Theorem~\ref{thm:the-one-object-case-cto}(a), if $A$ is in addition self-injective, then $\ch^{[-1,0]}(\proj A)$ admits a triangle structure. Conversely, if $A$ is any finite-dimensional $k$-algebra such that $\ch^{[-1,0]}(\proj A)$ admits a triangle structure, then $A$ is necessarily self-injective, see Appendix~\ref{a:triangle-structures-on-2-term-complexes} written by Osamu Iyama.
\begin{proof}
(a) 
By Theorem~\ref{thm:the-one-object-case}, $\mathbb{P}$ is an equivalence if and only if $\Hom_\ct(M,\Sigma^{-1}M)=0$. By \cite[Proposition 3.6]{IyamaOppermann13}, the latter condition is equivalent to that $A$ is self-injective.

(c)
It suffices to show that a relative cluster-tilting object in $\ct$ is cluster-tilting, which holds by \cite[Proposition 3.4]{YangZhu19}. 

(d) 
By (a), the functor $\mathbb{P}$ is an equivalence and $\Hom_\ct(M,\Sigma^{-1}M)=0$. Let $N$ be a cluster-tilting object of $\ct$. Assume that $\mathbb{P}(N)$ is tilting. Then by \cite[Theorem 2.1]{AlnofayeeRickard13}, the endomorphism algebra $\End(\mathbb{P}(N))$ is self-injective. Thus $\End(N)\cong\End(\mathbb{P}(N))$ is self-injective. Now assume that $\End(N)$ is self-injective. Then $\Hom(N,\Sigma^{-1}N)=0$, by \cite[Proposition 3.6]{IyamaOppermann13}. Therefore by Lemma~\ref{lem:P-on-negative-extensions}, we have $\Hom(\mathbb{P}(N),\Sigma^{-1}\mathbb{P}(N))=0$, implying that $\mathbb{P}(N)$ is tilting.
\end{proof}

\begin{remark}
\begin{itemize}
\item[(a)]
If $\ct$ is the cluster category of an acyclic quiver, then there are very few $M$ such that $A=\End_\ct(M)$ is self-injective, see \cite{Ringel08}. More examples are given in \cite{HerschendIyama11}.
\item[(b)]
Since $\ct$ is 2-Calabi--Yau, the two triangles \eqref{triangle:left-mutation} and \eqref{triangle:right-mutation} (although precisely one of them is an extriangle) produce the same cluster-tilting object (\cite[Section 5]{IyamaYoshino08}), so we simply denote the mutation by $\mu_i$.
\end{itemize}
\end{remark}

In the study of the normal form of $\ct$, an interesting problem is how much information of $\ct$ is carried by the algebra $A$.

\begin{corollary}
\label{cor:self-injective-cy-tilted-algebra}
Assume that $A$ is self-injective. Then $\ct$ is additively equivalent to $\ch^{[-1,0]}(\proj A)$.
\end{corollary}

As Lemma~\ref{lem:taking-presentation-preserves-extension-2} and Proposition~\ref{prop:P-detects-extriangles} show, $\ct$ and $\ch^{[-1,0]}(\proj A)$ are in fact equivalent as extriangulated categories.

\begin{corollary}
\label{cor:self-injective-representation-finite-cy-tilted-algebra}
Assume that $A$ is self-injective and that $\ct$ is connected and standard and has only finitely many isomorphism classes of indecomposable objects. Let $\ct'$ be another Hom-finite Krull--Schmidt algebraic $2$-Calabi--Yau triangulated $k$-category with a cluster-tilting object whose endomorphism algebra is isomorphic to $A$. Then $\ct$ and $\ct'$ are triangle equivalent.
\end{corollary}
\begin{proof}
By Corollary~\ref{cor:self-injective-cy-tilted-algebra}, $\ct$ and $\ct'$ are additively equivalent. Then the desired result follows from \cite[Corollary 2]{Keller18}.
\end{proof}

\subsection{Relation to support $\tau$-tilting theory}
\label{ss:relation-to-support-tau-tilting-theory}

Consider the restricted Yoneda functor
\[
\xymatrix@C=0.5pc@R=0.5pc{
\mathbb{H}_\ct\colon & \ct\ar[rrrr] &&&&\Mod\cm.\\
& X\ar@{|->}[rrrr] &&&& \Hom_\ct(?,X)|_\cm
}
\]
We have a commutative diagram
\[
\xymatrix@C=4pc{
\fpr(\cm)\ar[r]^{\mathbb{P}} \ar[dr]_{\mathbb{H}_\ct} & \ch^{[-1,0]}(\cm)\ar[d]^{\mathbb{H}_{\ch^b(\cm)}}\\
&\mod\cm,
}
\]
where $\mod\cm$ is the category of finitely presented functors $\cm\to\Mod k$.
Applying \cite[Theorem 3.4]{IyamaJorgensenYang14} and Theorem~\ref{thm:P-induces-bijection-between-cto} we recover \cite[Theorem 1.1]{ZhouZhu20}. But notice that \cite[Theorem 1.1]{ZhouZhu20} is more general.

\begin{theorem}
Assume that each object of $\ch^{[-1,0]}(\cm)$ can be written as a direct sum of indecomposable objects which are unique up to isomorphism. Then the map 
\[
\cn\mapsto (\mathbb{H}_\ct(\cn),\Sigma^{-1}\cm\cap\mathbb{H}_\ct(\cn))
\]
is a bijection from the first of the following sets to the second:
\begin{itemize}
\item[(1)] relative rigid subcategories of $\fpr(\cm)$,
\item[(2)] $\tau$-rigid pairs of $\mod\cm$.
\end{itemize}
It restricts to a bijection from the first of the following sets to the second:
\begin{itemize}
\item[(3)] weakly relative cluster-tilting subcategories of $\fpr(\cm)$,
\item[(4)] support $\tau$-tilting pairs of $\mod\cm$.
\end{itemize}  
\end{theorem}

\section{Quivers with potential and Mizuno's theorem}

In this section we apply Theorem~\ref{thm:the-one-object-case-cto} to give an alternative proof of the beautiful theorem \cite[Theorem 1.1(a)]{Mizuno15} of Mizuno, which is about the compatibility of silting mutation in the homotopy category and mutation of quivers with potential for self-injective quivers with potential.

\smallskip
Assume that $k$ is a field.

\subsection{Quivers with potential and their cluster categories}

Let $Q$ be a finite quiver with vertex set $Q_0$ and arrow set $Q_1$. A {\em potential} on $Q$ is a (possibly infinite) linear combination of non-trivial cyclic paths of $Q$.
The \emph{complete Jacobian algebra} of a quiver with potential $(Q,W)$ is the quotient of the complete path algebra $\widehat{kQ}$ by the closure of the ideal generated by the cyclic derivatives $\del_\alpha W$, where $\alpha$ runs over all arrows of $Q$:
\[
\widehat{J}(Q,W)=\widehat{kQ}/\overline{(\del_\alpha W:\alpha\in Q_1)}.
\]
For a vertex $i$ of $Q$ not lying no a loop or 2-cycle, one can construct a new quiver with potential $\mu_i(Q,W)$, called the \emph{mutation} of $(Q,W)$. This operation extends quiver mutation. See \cite{DerksenWeymanZelevinsky08}. 

\smallskip
Let $(Q,W)$ be a quiver with potential such that $\widehat{J}(Q,W)$ is finite-dimensional. Then there is a category $\cc_{(Q,W)}$, called the \emph{cluster category} of $(Q,W)$, 
which is a Hom-finite Krull--Schmidt 2-Calabi--Yau algebraic triangulated category, and which has a basic cluster-tilting object $T_{(Q,W)}$ with endomorphism algebra isomorphic to $\widehat{J}(Q,W)$, see \cite[Theorem 2.1]{Amiot09} and \cite[Theorem A.21]{KellerYang11} (see also \cite[Theorem 5.8]{IyamaYang18}).

\begin{lemma}
\label{lem:compatibility-of-mutation-of-QP-and-cto}
For a sequence $(i_1,\ldots,i_s)$ of vertices of $Q$ such that $i_{r}$ does not lie on a loop or 2-cycle in the quiver of $\mu_{i_r-1}\cdots\mu_{i_1}(Q,W)$ for all $1\leq r\leq s-1$, we have 
\[
\End_{\cc_{(Q,W)}}(\mu_{i_s}\cdots\mu_{i_1}T_{(Q,W)})\cong \widehat{J}(\mu_{i_s}\cdots\mu_{i_1}(Q,W)).
\]
\end{lemma}
\begin{proof}
Let $i$ be a vertex of $(Q,W)$ which does not lie on a loop or a 2-cycle. 
By the paragraph before \cite[Proposition 4.5]{KellerYang11}, there is a triangle equivalence
\[
\cc_{\mu_i(Q,W)}\longrightarrow \cc_{(Q,W)},
\]
which takes $T_{\mu_i(Q,W)}$ to $\mu_iT_{(Q,W)}$. The desired result then follows by induction.
\end{proof}

A quiver with potential is said to be \emph{self-injective} if its Jacobian algebra is a finite-dimensional self-injective $k$-algebra. Examples of self-injective quivers with potential are constructed in \cite{HerschendIyama11}. The following two results are immediate consequences of Theorem~\ref{thm:the-one-object-case-cto} and Corollaries~\ref{cor:self-injective-cy-tilted-algebra} and \ref{cor:self-injective-representation-finite-cy-tilted-algebra}.

\begin{corollary}
Let $(Q,W)$ be a self-injective quiver with potential. Then the category $\ch^{[-1,0]}(\proj \widehat{J}(Q,W))$ admits a triangle structure.
\end{corollary}

\begin{corollary}
Let $\ct$ be a Hom-finite Krull--Schmidt 2-Calabi--Yau algebraic triangulated category with a cluster-tilting object whose endomorphism algebra is isomorphic to $\widehat{J}(Q,W)$ for some self-injective quiver with potential $(Q,W)$.
Then $\ct$ is additively equivalent to $\cc_{(Q,W)}$. If in addition $\ct$ is connected,  standard and has only finitely many isomorphism classes of indecomposable objects, then $\ct$ is triangle equivalent to $\cc_{(Q,W)}$.
\end{corollary}

This corollary shows that in this case Amiot's conjecture hods (see the second question in \cite[Summary of results, Part 2, Perspectives]{Amiot08} and \cite[Conjecture 1.1]{KalckYang20}). See \cite{LiuKeller23} for a recent progress on Amiot's conjecture.

\subsection{Mizuno's result}

Let $(Q,W)$ be a self-injective quiver with potential and put $J=\widehat{J}(Q,W)$.

\begin{theorem}\label{thm:compatibility-with-mutations}
Let  $\underline{i}=(i_1,\ldots,i_s)$ be a sequence of vertices of $Q$ such that $i_{r}$ does not lie on a loop or 2-cycle in the quiver of $\mu_{i_r-1}\cdots\mu_{i_1}(Q,W)$ for all $1\leq r\leq s-1$ and let $\underline{\epsilon}=(\epsilon_1,\ldots,\epsilon_s)$ be a sequence of elements of $\{L,R\}$ such that one of the following conditions is satisfied
\begin{itemize}
\item[(1)] $J_{\underline{i},\underline{\epsilon}}=\mu_{i_r}^{\epsilon_r}\cdots\mu_{i_1}^{\epsilon_1}(J)$ belongs to $\ch^{[-1,0]}(\proj J)$ for any $r\in\{1,\ldots,s\}$, 
\item[(2)] $J_{\underline{i},\underline{\epsilon}}=\mu_{i_r}^{\epsilon_r}\cdots\mu_{i_1}^{\epsilon_1}(\Sigma J)$ belongs to $\ch^{[-1,0]}(\proj J)$ for any $r\in\{1,\ldots,s\}$.
\end{itemize} Then 
\[
\End_{\ch^b(\proj J)}(J_{\underline{i},\underline{\epsilon}})\cong \widehat{J}(\mu_{i_s}\cdots\mu_{i_1}(Q,W)).
\]
\end{theorem}
\begin{proof}
Since $J$ is self-injective, it follows from Theorem~\ref{thm:the-one-object-case-cto} that the functor $\mathbb{P}\colon\cc_{(Q,W)}\to \ch^{[-1,0]}(\proj J)$ is an equivalence which takes $T_{(Q,W)}$ to $J$, $\Sigma T_{(Q,W)}$ to $\Sigma J$, and which takes a basic cluster-tilting object to a basic 2-term silting object and commutes with mutations. Therefore in both cases we have
\[
\End_{\ch^b(\proj J)}(J_{\underline{i},\underline{\epsilon}})\cong \End_{\cc_{(Q,W)})}(\mu_{i_s}\cdots\mu_{i_1}(T_{(Q,W)}),
\]
which is isomorphic to $\widehat{J}(\mu_{i_s}\cdots\mu_{i_1}(Q,W))$ by Lemma~\ref{lem:compatibility-of-mutation-of-QP-and-cto}.
\end{proof}

We remark that if the condition (1) or (2) is assumed, then the sequence $\underline{\epsilon}$ of signs is uniquely determined. The whole study in this paper was motivated by the following theorem due to Mizuno. For a set $I$ of vertices of $Q$ satisfying the two conditions below, we define $\mu_I=\prod_{i\in I}\mu_i$ for all mutations appearing below.

\begin{theorem}[{\cite[Theorem 1.1(a)]{Mizuno15}}]
Let  $I$ be a set of vertices of $Q$ satisfying
\begin{itemize}
\item any $i\in I$ does not lie on a loop or 2-cycle,
\item there are no arrows between vertices in $I$.
\end{itemize} 
Then 
\[
\End_{\ch^b(\proj J)}(\mu_I^L(J))\cong \widehat{J}(\mu_{I}(Q,W))\cong \End_{\ch^b(\proj J)}(\mu_I^R(\Sigma J)).
\]
\end{theorem}
\begin{proof}
Suppose $I=\{i_1,\ldots,i_s\}$. Applying Theorem~\ref{thm:compatibility-with-mutations} to $((i_1,\ldots,i_s),(L,\ldots,L))$ and $((i_1,\ldots,i_s),(R,\ldots,R))$, we obtain the desired result.
\end{proof}

\appendix
\section{Triangle structures on 2-term complexes\\ by Osamu Iyama}
\label{a:triangle-structures-on-2-term-complexes}

Let $A$ be a ring, and $\proj A$ the category of finitely generated projective $A$-modules.
We denote by $\ch(\proj A)$ the homotopy category of $\proj A$, and by $\ch^{[-1,0]}(\proj A)$ the full subcategory consisting of two-term complexes.
The aim of this appendix is to prove the converse of Corollary \ref{cor:self-injective-cy-tilted-algebra}.
More precisely, we prove the following result.

\begin{theorem}\label{3 and 4}
Let $\ct$ be a 2-Calabi--Yau triangulated category over a perfect field $k$, $M$ a cluster-tilting object and $A=\End_{\ct}(M)$.
Then the following conditions are equivalent.
\begin{enumerate}
\item[\rm(i)] $\ch^{[-1,0]}(\proj A)$ has a structure of a triangulated category.
\item[\rm(ii)] $A$ is self-injective.
\item[\rm(iii)] $\add M=\add M[2]$.
\item[\rm(iv)] $\proj A$ has a structure of a 4-angulated category.
\end{enumerate}
\end{theorem}

The novel part of Theorem \ref{3 and 4} is the implication (i)$\Rightarrow$(ii)(iv), which is a consequence of the following main result of this section.

\begin{theorem}\label{necessary}
Let $A$ be a finite dimensional algebra over a field $k$ such that $\ch^{[-1,0]}(\proj A)$ has a structure of a triangulated category.
Then the following assertions hold.
\begin{enumerate}
\item[\rm(a)] $A$ is self-injective.
\item[\rm(b)] Let $S$ be a simple $A$-module which is non-projective. Then $\Omega^4(S)$ is again a simple $A$-module.
\end{enumerate}
If $A/{\rm rad} A$ is a separable $k$-algebra, then the following assertions hold.
\begin{enumerate}
\item[\rm(c)] $A$ is \emph{twisted 4-periodic}, that is, there exists an automorphism $\alpha$ of $A$ such that
$\Omega^4_{A^{\rm e}}(A)\simeq {}_1A_\alpha$ as $A^e$-modules.
\item[\rm(d)] $\proj A$ has a structure of a 4-angulated category.
\end{enumerate}
\end{theorem}

%It is well-known that, if $A/J_A$ is a separable $k$-algebra (e.g.\ $k$ is algebraically closed), then the condition (b) is equivalent to the following conditions  (see e.g.\ \cite[Corollary 2.2]{H})
%\begin{enumerate}
%\item[\rm(c)] There exists an automorphism $\alpha$ of $A$ such that $\Omega^4\simeq(-)_\alpha$ as functors on $\underline{\mod}A$.
%\item[\rm(d)] There exists an automorphism $\alpha$ of $A$ such that
%$\Omega^4_{A^{\rm e}}(A)\simeq {}_1A_\alpha$ as $A^e$-modules.
%\end{enumerate}
%The following is quite interesting even though it is more reasonable to impose suitable assumptions.
It is natural to pose the following question.

\begin{question}
Let $\cm$ be a 4-angulated category.
Does $\ch^{[-1,0]}(\cm)$ have a structure of a triangulated category?
\end{question}

In the rest, we prove Theorems \ref{necessary} and \ref{3 and 4}.
Let us recall basic notions.
%In particular, we have a bijection
%\[\ind(\mod A)\sqcup\ind(\proj A)\simeq\ind\ch^{[-1,0]}(\proj A),\]
%where $X\in\ind(\mod A)$ corresponds to its minimal projective presentation, and $P\in\ind(\proj A)$ corresponds to $P[1]$.
%We often regard $\ind(\mod A)$ as a subset of $\ind\ch^{[-1,0]}(\proj A)$.

Let $\cc$ be a Krull--Schmidt category, and $f:Y\to X$ a morphism in $\cc$.
We call $f$ \emph{right minimal} if each morphism $g:Y\to Y$ satisfying $f=fg$ is an automorphism. We call $f$ \emph{right almost split} if $f$ is not a split epimorphism, and for each morphism $g:Z\to X$ which is not a split epimorphism, there exists $h:Z\to Y$ satisfying $g=fh$.
In this case, it is easy to show that $X$ is indecomposable.
We call $f$ a \emph{sink morphism of $X$} if it is right minimal and right almost split morphism. Dually, we define a \emph{left minimal} morphism, a \emph{left almost split} morphism and a \emph{source} morphism.

Let $\cc$ be a Krull--Schmidt triangulated category. A triangle $Z\xrightarrow{g}Y\xrightarrow{f}X\to\Sigma Z$ is called an \emph{almost split triangle} if $f$ is a sink morphism, or equivalently, $g$ is a source morphism, see \cite[Theorem 2.9]{IyamaNakaokaPalu24} for other equivalent conditions.
%Dually we define a \emph{left $\cd$-approximation}.
%We call $\cd$ \emph{functorially finite in $\cc$} if each object of $\cc$ admits a right $\cd$-approximation and a left $\cd$-approximation.

%We start with the following standard facts.
%\begin{proposition}
%Let $A$ be a finite dimensional algebra over a field $k$.
%\begin{enumerate}
%\item[\rm(a)] Any indecomposable object in $\ch^{[-1,0]}(\proj A)$ has a sink (respectively, source) morphism in $\ch^{[-1,0]}(\proj A)$.
%\item[\rm(b)] If $\ch^{[-1,0]}(\proj A)$ has a structure of a triangulated category, then it has almost split triangles in $\ch^{[-1,0]}(\proj A)$.
%\end{enumerate}
%\end{proposition}
%\begin{proof}
%(a) Each indecomposable object $X\in\ch^{[-1,0]}(\proj A)$ has an almost split triangle
%\[\nu(X)[-1]\to E\xrightarrow{f} X\to \nu(X)\ \ {\rm in }\ \cd^b(\mod A),\]
%where $\nu:=-\otimes_A(DA)$ is the Nakayama functor \cite[proof of Proposition 5.18]{IyamaNakaokaPalu24}. Since $\ch^{[-1,0]}(\proj A)$ is functorially finite in $\cd^b(\mod A)$ \cite[Example 5.19]{IyamaNakaokaPalu24}, there exists a right $\ch^{[-1,0]}(\proj A)$-approximation $g:F\to E$ of $E$. Then $fg:F\to X$ is right almost split in $\ch^{[-1,0]}(\proj A)$. Taking its right minimal part, we obtain a sink morphism of $X$ in $\ch^{[-1,0]}(\proj A)$. Dually, $X$ has a source morphism in $\ch^{[-1,0]}(\proj A)$.
%(b) Immediate from (a) and \cite[Lemma 3.2]{IyamaNakaokaPalu24}.
%\end{proof}

Later we use the following characterization of almost split triangles.

\begin{lemma}\label{tau sequence}
Let $\ct$ be a Krull-Schmidt triangulated category.
Assume that $X\in\ct$ is indecomposable and satisfies $({\rm rad}\ct)(-,X)\neq0$.
Then the following conditions are equivalent for a sequence $Z\xrightarrow{g} Y\xrightarrow{f} X$.
\begin{enumerate}
\item[\rm(i)] $f$ is a sink morphism of $X$, and $g$ is a right minimal weak kernel of $f$.
\item[\rm(ii)] There exists a morphism $e:X\to\Sigma Z$ such that $Z\xrightarrow{g} Y\xrightarrow{f} X\xrightarrow{e}\Sigma Z$ is an almost split triangle.
\end{enumerate}
\end{lemma}

\begin{proof}
(ii)$\Rightarrow$(i) 
It suffices to show that $g$ is right minimal. Since $g:Z\to Y$ is a source morphism, $Z$ is indecomposable. If $g$ is not right minimal, then $g$ has to be zero, and hence $f$ has to be a split monomorphism. Since $f$ belongs to the radical of $\ct$, we have $f=0$ and hence $({\rm rad}\ct)(-,X)=0$, a contradiction.

(i)$\Rightarrow$(ii)
We extend $f$ to an almost split triangle $Z'\xrightarrow{g'} Y\xrightarrow{f} X\xrightarrow{e'}\Sigma Z'$. By the argument above, $g'$ is also right minimal. Since both $g$ and $g'$ are right minimal weak kernels of $f$, there exists an isomorphism $a:Z\to Z'$ such that $g=g'a$. Thus the condition (ii) holds for $e:=(\Sigma a)^{-1}e'$.
\end{proof}

Let $\cc$ be an additive category and $\cd$ an additive full subcategory of $\cc$.
We call a morphism $f:Y\to X$ a \emph{right $\cd$-approximation of $X$} if $Y\in\cd$ holds, and for each morphism $g:Z\to X$ with $Z\in\cd$, there exists $h:Z\to Y$ satisfying $g=fh$.

Let $A$ be a finite dimensional $k$-algebra. By \cite{Auslander71}, we have an equivalence \begin{equation}\label{H^0}
H^0:\ch^{[-1,0]}(\proj A)/[A[1]]\simeq \mod A,
\end{equation}
where $[A[1]]$ is the ideal of $\ch^{[-1,0]}(\proj A)$ consisting of morphisms factoring through objects in $\add A[1]$.

The following observations play a crucial role in the proof of Theorem \ref{3 and 4}.

\begin{lemma}\label{compare 3 and 4}
Let $A$ be a finite dimensional $k$-algebra. For a complex
\begin{equation}\label{4 term sequence}
P_3\xrightarrow{f_3} P_2\xrightarrow{f_2}P_1\xrightarrow{f_1} P_0\ \ {\rm in }\ \proj A,
\end{equation}
consider the complex
\begin{equation}\label{3 term sequence}
P_3[1]\xrightarrow{f'_3} C\xrightarrow{f'_1} P_0\ \ {\rm in }\ \ch^{[-1,0]}(\proj A),
\end{equation}
where $C={\rm Cone}(P_2\xrightarrow{f_2}P_1)\in\ct$ and $f'_1$ and $f'_3$ are given by
\begin{equation}\label{f'_3 and f'_1}
\xymatrix@R1.5em{
P_3\ar[r]^{f_3}\ar[d]&P_2\ar[r]\ar[d]^{f_2}&0\ar[d]\\
0\ar[r]&P_1\ar[r]^{f_1}&P_0.
}\end{equation}
Then the following assertions hold.
\begin{enumerate}
\item[\rm(a)] The following conditions are equivalent.
\begin{enumerate}
\item[\rm(i)] \eqref{4 term sequence} is exact.
\item[\rm(ii)] $f'_3:P_3\to C$ is a weak kernel of $f'_1:C\to P_0$ in $\ch^{[-1,0]}(\proj A)$.
\end{enumerate}
\item[(b)] The following conditions are equivalent.
\begin{enumerate}
\item[\rm(i)] $f_3:P_3\to P_2$ is right minimal in $\proj A$.
\item[\rm(ii)] $f'_3:P_3[1]\to C$ is right minimal in $\ch^{[-1,0]}(\proj A)$.
\end{enumerate}
%\item[\rm(c)] If the conditions in {\rm (a)} hold, then the following conditions are equivalent.
%\begin{enumerate}
%\item[\rm(i)] ${\rm Cok} f_1$ is a simple $A$-module.
%\item[\rm(ii)] $f'_1:C\to P_0$ is right minimal almost split in $\ch^{[-1,0]}(\proj A)$.
%\end{enumerate}
\end{enumerate}
\end{lemma} 

\begin{proof}
For simplicity, write $\ch:=\ch(\proj A)\supset\ct:=\ch^{[-1,0]}(\proj A)$.

(a)(ii)$\Rightarrow$(i) Applying $\Hom_{\ch}(A,-)$ to the sequence \eqref{3 term sequence}, we obtain an exact sequence
\[0=\Hom_{\ch}(A,P_3[1])\xrightarrow{f'_3}\overbrace{\Hom_{\ch}(A,C)}^{{\rm Cok} f_2=}\xrightarrow{f'_1}\Hom_{\ch}(A,P_0)=P_0\]
Thus the morphism ${\rm Cok} f_2\to P_0$ induced by $f_1$ is injective, and hence the sequence \eqref{4 term sequence} is exact at $P_1$. Similarly, applying $\Hom_{\ch}(A[1],-)$ to the sequence \eqref{3 term sequence}, we obtain an exact sequence
\[P_3=\Hom_{\ch}(A[1],P_3[1])\xrightarrow{f'_3}\overbrace{\Hom_{\ch}(A[1],C)}^{{\rm Ker} f_2=}\xrightarrow{f'_1}\Hom_{\ch}(A[1],P_0)=0\]
Thus the morphism $P_3\to {\rm Ker} f_2$ induced by $f_3$ is surjective, and hence the sequence \eqref{4 term sequence} is exact at $P_2$.

(i)$\Rightarrow$(ii) Take a triangle
\begin{equation}\label{DC triangle}
D\xrightarrow{g} C\xrightarrow{f'_1}P_0\to D[1]\ \ \rm{ in }\ \ch.
\end{equation}
Then $f'_3:P_3[1]\to C$ can be written as a composition
\[\xymatrix@R1.5em{
P_3[1]\ar[d]^{f''_3}&P_3\ar[r]\ar[d]^{f_3}&0\ar[r]\ar[d]&0\ar[d]\\
D\ar[d]^g&P_2\ar[r]^{f_2}\ar[d]^1&P_1\ar[r]^{f_1}\ar[d]^1&P_0\ar[d]\\
C&P_2\ar[r]^{f_2}&P_1\ar[r]&0}\]
By \eqref{DC triangle}, $g$ is a weak kernel of $f'_1$ in $\ch$. Thus it suffices to show that $f''_3:P_3[1]\to D$ is a right $\ct$-approximation of $D$. Let
\[E:={\rm Cone}(f''_3:P_3[1]\to D)=[P_3\xrightarrow{f_3} P_2\xrightarrow{f_2}P_1\xrightarrow{f_1} P_0],\]
where $P_0$ is in degree $1$. Then $H^i(E)=0$ holds for $i=0,-1$ since \eqref{4 term sequence} is exact, that is, $\Hom_{\ch}(A,E)=0=\Hom_{\ch}(A[1],E)$. Thus we obtain $\Hom_{\ch}(\ct,E)=0$. Since we have a triangle
$P_3[1]\xrightarrow{f''_3} D\to E\to P_3[2]$ in $\ch$, it follows that $f''_3:P_3[1]\to D$ is a right $\ct$-approximation of $D$, as desired.

%Since \eqref{ass of P} is a triangle in $\ch(\proj A)$, we have an exact sequence
%\[\Hom_{\ch(\proj A)}(-,P_3[1])\xrightarrow{f'_3}\Hom_{\ch(\proj A)}(-,C)\xrightarrow{f'_1}\Hom_{\ch(\proj A)}(-,P_1).\]
%Evaluating objects of $\ch$, it follows that $f'_3$ is a weak 
%For each $Q=(Q_1\to Q_0)\in\ch$, an easy diagram chase shows that $\Hom_{\ch}(Q,P_3)\xrightarrow{f_3}\Hom_{\ch}(Q,E)\xrightarrow{f_1}\Hom_{\ch}(Q,P_1)$ is exact.
%Thus $f_3:P_3\to E$ is a weak kernel of $f_1:E\to P_0$.
%Moreover, $f_1:E\to P_0$ is right minimal since $f_3:P_3\to E$ is in the radical of $\ch$.

(b) Consider the following conditions:
\begin{enumerate}
\item[\rm(i$'$)] For each non-zero direct summand $X$ of $P_3$, the composition $X\to P_3\xrightarrow{f'_3}P_2$ is non-zero.
\item[\rm(ii$'$)] For each non-zero direct summand $Y$ of $P_3[1]$, the composition $Y\to P_3[1]\xrightarrow{f'_3}C$ is non-zero.
\end{enumerate}
Then clearly we have (i)$\Leftrightarrow$(i$'$)$\Leftrightarrow$(ii$'$)$\Leftrightarrow$(ii).%Otherwise, there exists a non-zero direct summand $X$ of $P_3[1]$ such that the composition $X\to P_3[1]\xrightarrow{f'_3}C$ is zero. Write $X=Q[1]$ with $Q\in\proj A$. Then $Q$ is a non-zero direct summand of $P_3$ such that the composition $Q\to P_3\xrightarrow{f_3}P_2$ is zero. This is a contradiction since $f_3:P_3\to P_2$ is right minimal by our assumption.
\end{proof}

We also need the following observation.

\begin{lemma}\label{compare 3 and 4 1}
Let $A$ be a finite dimensional $k$-algebra, and let
\begin{equation*}%\label{4 term sequence 1}
P_2\xrightarrow{f_2}P_1\xrightarrow{f_1} P_0
\end{equation*}
be an exact sequence in $\mod A$ such that $P_i\in\proj A$ holds for each $i$, $P_0$ is indecomposable and $f_1$ is not surjective. Then the following conditions are equivalent.
\begin{enumerate}
\item[\rm(i)] ${\rm Im} f_1={\rm rad} P_0$ holds.
%(or equivalently, $f_1:P_1\to P_0$ is right almost split in $\proj A$).
\item[\rm(ii)] The morphism $f'_1:C:={\rm Cone}(P_2\xrightarrow{f_2}P_1)\to P_0$ given by
\begin{equation*}
\xymatrix@R1.5em{
P_2\ar[r]\ar[d]_{f_2}&0\ar[d]\\
P_1\ar[r]^{f_1}&P_0
}\end{equation*}
is right almost split in $\ch^{[-1,0]}(\proj A)$.
%\begin{equation}\label{3 term sequence 0}
%C\xrightarrow{f'_1} P_0\ \ {\rm in }\ \ch^{[-1,0]}(\proj A),
%\end{equation}
%where $C=\in\ct$ and $f'_1$ is given by
\end{enumerate}
\end{lemma} 

\begin{proof}
%We prove that $f'_1: C\to P_0$ is right almost split in $\ct$.
For simplicity, write $\ct:=\ch^{[-1,0]}(\proj A)$. We have a commutative diagram
\begin{equation}\label{square}
\xymatrix{
\Hom_{\ct}(-,C)\ar[r]^{f'_1}\ar@{->>}[d]^{H^0}&({\rm rad}\ct)(-,P_0)\ar[d]^{H^0}_\wr\\
\Hom_A(H^0(-),{\rm Cok} f_2)\ar[r]&({\rm rad}(\mod A))(H^0(-),P_0),
}\end{equation}
where the left vertical map is surjective and the right one is bijective since the functor \eqref{H^0} is an equivalence and $\Hom_{\ct}(A[1],P_0)=0$ holds.

Thus (ii) holds if and only if the upper map in \eqref{square} is surjective if and only if the lower map in \eqref{square} is surjective if and only if the morphism ${\rm Cok} f_2\to P_0$ induced by $f_1$ is right almost split in $\mod A$.
The last condition is equivalent to (i) since the inclusion ${\rm rad} P_0\to P_0$ is a sink morphism of $P_0$ in $\mod A$.
 \end{proof}

Combining the previous observations, we obtain the following result.
%We consider almost split triangles ending at indecomposable projective $A$-modules.

\begin{proposition}\label{right almost split of P}
Assume that $\ct=\ch^{[-1,0]}(\proj A)$ has a structure of a triangulated category with suspension functor $\Sigma$.
For each non-projective simple $A$-module $S$, we take a minimal projective resolution 
\begin{equation*}
\cdots\to P_3\xrightarrow{f_3} P_2\xrightarrow{f_2}P_1\xrightarrow{f_1} P_0\to S\to0.
\end{equation*}
Then there exists a morphism $e:P_0\to\Sigma(P_3[1])$ such that
\begin{equation}\label{ass of P}
P_3[1]\xrightarrow{f'_3} C\xrightarrow{f'_1} P_0\xrightarrow{e}\Sigma(P_3[1])
\end{equation}
is an almost split triangle in $\ct$, where $C={\rm Cone}(P_2\xrightarrow{f_2}P_1)\in\ct$ and $f'_1$ and $f'_3$ are given by \eqref{f'_3 and f'_1}.
%\[\xymatrix@R1.5em{
%P_3\ar[r]^{f_3}\ar[d]&P_2\ar[r]\ar[d]^{f_2}&0\ar[d]\\
%0\ar[r]&P_1\ar[r]^{f_1}&P_0
%}\]
%In particular, $P_3$ is indecomposable.
\end{proposition} 

\begin{proof}
By Lemma \ref{compare 3 and 4 1}, $f'_1:C\to P_0$ is right almost split. By Lemma \ref{compare 3 and 4}(a)(b), $f'_3:P_3[1]\to C$ is a right minimal weak kernel of $f'_1$.
On the other hand, since $S$ is non-projective, $({\rm rad}\ct)(-,P_0)\neq0$ holds.
By Lemma \ref{tau sequence}(i)$\Rightarrow$(ii), the assertion follows.
\end{proof}

Now we are ready to prove Theorem \ref{necessary}.

\begin{proof}[Proof of Theorem \ref{necessary}]
Assume that $\ch^{[-1,0]}(\proj A)$ has a structure of a triangulated category with suspension functor $\Sigma$.

(i) For the opposite algebra $A\op$ of $A$, the duality $(-)^*:=\Hom_A(-,A):\proj A\simeq\proj A\op$ gives a duality
\begin{equation}\label{2 duality}
\Hom_A(-,A[1]):\ch^{[-1,0]}(\proj A)\simeq\ch^{[-1,0]}(\proj A\op).
\end{equation}
In particular,  $\ch^{[-1,0]}(\proj A\op)$ also has a structure of a triangulated category.

(ii) We show that any simple non-projective $A$-module $S$ satisfies $\Ext^i_A(S,A)=0$ for $i=1,2$ and ${\rm Tr}\Omega^2(S)$ is a simple $A\op$-module.
%\begin{enumerate}
%\item[\rm(a)] $S$ is a projective $A$-module.
%\item[\rm(b)] $\Ext^1_A(S,A)
%\end{enumerate}

By Proposition \ref{right almost split of P}, we have an almost split triangle \eqref{ass of P} in $\ch^{[-1,0]}(\proj A)$. In particular, $P_3$ is indecomposable.
Since the notion of almost split triangles is self-dual, by applying the duality \eqref{2 duality}, we obtain an almost split triangle
\begin{equation}\label{ass of P 2}
P_0^*[1]\xrightarrow{f^*_1{}'} C'\xrightarrow{f^*_3{}'} P_3^*\to\Sigma(P_0^*[1])
\end{equation}
in $\ch^{[-1,0]}(\proj A\op)$, where $C':=(P_1^*\xrightarrow{f_2^*}P_2^*)$.
Applying Lemma \ref{compare 3 and 4}(a)(ii)$\Rightarrow$(i) to the algebra $A^*$ and the sequence \eqref{ass of P 2}, it follows that the sequence
\[P_0^*\xrightarrow{f_1^*} P_1^*\xrightarrow{f_2^*} P_2^*\xrightarrow{f_3^*} P_3^*\]
is exact. Thus $\Ext^i_A(S,A)=0$ holds for $i=1,2$.
Applying Lemma \ref{compare 3 and 4 1}(a)(ii)$\Rightarrow$(i) to the sequence \eqref{ass of P 2}, we obtain ${\rm Im} f_3^*={\rm rad} P_3^*$. Since $P_3^*$ is indecomposable, ${\rm Tr}\Omega^2(S)={\rm Cok} f_3^*={\rm top} P_3^*$ is a simple $A\op$-module.
Thus the assertions follow.

(iii) We prove (a) and (b) in Theorem \ref{necessary}.

First, we prove (a), that is, $A$ is self-injective. It suffices to show $\Ext^1_A(S,A)=0$ for each simple $A$-module $S$. If $S$ is projective, then this is clear. Otherwise, this follows from (ii). Thus the assertion follows.

Next, we prove (b). Let $S$ be a simple non-projective $A$-module. By (ii), ${\rm Tr}\Omega^2S$ is a simple $A\op$-module. Since $A$ is self-injective, $(-)^*=\Hom_A(-,A):\mod A\op\to\mod A$ is an exact duality, and hence $({\rm Tr}\Omega^2(S))^*$ is a simple $A$-module.
Since $A$ is self-injective, we have $\Omega^4=({\rm Tr}\Omega^2-)^*$. Thus $\Omega^4S$ is a simple $A$-module.

(iv) We prove (c) and (d) in Theorem \ref{necessary}.

By (b) and the separability of the $k$-algebra $A/J_A$, we obtain (c) by \cite[Theorem 1.4]{GreenSnashallSolberg03} and \cite[Corollary 2.2]{H} (see also \cite[Proposition 3.5]{ChanDarpoeIyamaMarczinzik25}).
By (c) and \cite[Theorem 1.3]{Lin19}, we obtain (d).
\end{proof}

Finally, we prove Theorem \ref{3 and 4}.

\begin{proof}[Proof of Theorem \ref{3 and 4}]
(ii)$\Leftrightarrow$(iii) follows from \cite[Proposition 3.6]{IyamaOppermann13}.

(iv)$\Rightarrow$(ii) follows from \cite[Proposition 2.2.9]{JassoKellerMuro}.

(ii)$\Rightarrow$(i) follows from Corollary \ref{cor:self-injective-cy-tilted-algebra}.

(i)$\Rightarrow$(iv) follows from Theorem \ref{necessary}(d).

Alternatively, (iii)$\Rightarrow$(iv) follows from \cite[Theorem 1]{GeissKellerOppermann13}
\end{proof}

\def\cprime{$'$}
\providecommand{\bysame}{\leavevmode\hbox to3em{\hrulefill}\thinspace}
\providecommand{\MR}{\relax\ifhmode\unskip\space\fi MR }
\providecommand{\MRhref}[2]{%
  \href{http://www.ams.org/mathscinet-getitem?mr=#1}{#2}
}
\providecommand{\href}[2]{#2}

\end{document}